\documentclass[12pt]{article}

\usepackage{tikz,pgf}
\usepackage{float}
\usepackage{amssymb}
\usepackage{amsthm}
\usepackage{amsmath}	
\usepackage{mathrsfs}
\usepackage{enumitem}
\usepackage{authblk}

\newtheorem{thm}{Theorem}
\newtheorem{prop}{Proposition}
\newtheorem{lem}{Lemma}
\newtheorem{cor}{Corrollary}
\newtheorem{dfn}{Definition}
\newtheorem{exm}{Example}
\newtheorem{rmk}{Remark}

\title{
 Non-isomorphic Cayley Graphs with
 Same Random Walk Distributions
}

\author[1]{
Masao Ishikawa
}
\author[2]{
Fumihiko Nakano
}
\author[3]{
Taizo Sadahiro
}

\affil[1]{Department of Mathematics, Okayama University}
\affil[2]{Mathematical Institute, Tohoku University}
\affil[3]{Department of Computer Science, Tsuda University}

\date{}
\begin{document}

\maketitle

\begin{abstract}
 We construct an infinite family of
 triples $(G, S_1, S_2)$ each consisting of a group $G$
 and a pair $(S_1, S_2)$ of distinct subsets
 of $G$ with the following properties.
 \begin{enumerate}[font=\bfseries, label=\roman*]
  \item The two Cayley graphs 
	${\tt Cay}(G, S_1)$ and ${\tt Cay}(G, S_2)$ are non-isomorphic.
  \item 
	The distributions of the simple random walks on 
	${\tt Cay}(G, S_1)$ and ${\tt Cay}(G, S_2)$
	are the same  if one applies an
	appropriate bijection between the two vertex sets
	at each step.
  \item 
	The spectral set of ${\tt Cay}(G, S_i)$
	is decomposed into a disjoint union of two subsets
	$A$ and $B_i$ of the equal size ($|A|=|B_i|=|G|/2$),
	which satisfies
	$B_2=-B_1=\left\{-\lambda\,|\,\lambda\in B_1\right\}$.
 \end{enumerate}
 As a byproduct, 
 an infinite family of the pairs of isomorphic Cayley graphs
 on non-isomorphic groups is obtained.
\end{abstract}


\section{Introduction}
The simple random walks on non-isomorphic Cayley graphs generally 
have different total variation distances from the uniform distribution.
However, rare exceptions exist.
We present an infinite family of such exceptions, i.e., 
non-isomorphic pairs of Cayley graphs
whose simple random walks have exactly the same total variation distance
from the uniform distribution at each step.
Each pair is given as a pair of distinct quotient graphs of the same graph,
whose construction  is similar to the Terras's construction
 \cite{terras2010zeta}
of the isospectral non-isomorphic graphs.
(See also \cite{buser2006cayley} and \cite{brooks1999sunada}
for the details of the construction of isospectral graphs.)
However, the pairs we construct are not isospectral pairs.
Their spectral sets are half the identical and half opposite.

To be more precise, the main aim of this paper is to give a simple and  explicit construction of  an infinite
family of triples $(G, S_1, S_2)$ each consisting of a group $G$
 and a pair
 $(S_1, S_2)$ of distinct subsets
 of $G$ with the following properties.
 \begin{enumerate}[font=\bfseries, label=\roman*]
  \item \label{cond:nonisomorphic}
	The two Cayley graphs 
	${\tt Cay}(G, S_1)$ and ${\tt Cay}(G, S_2)$ are non-isomorphic.
  \item \label{cond:RW}
	The distributions of the simple random walks on 
	${\tt Cay}(G, S_1)$ and ${\tt Cay}(G, S_2)$
	are the same 
	if one applies an appropriate bijection between the two  vertex sets
	at each step.
  \item \label{cond:spectral}
	The spectral set of ${\tt Cay}(G, S_i)$
	is decomposed into a disjoint union of two subsets
	$A$ and $B_i$ of the equal size ($|A|=|B_i|=|G|/2$)
	which satisfies
	$B_2=-B_1=\left\{-\lambda\,|\,\lambda\in B_1\right\}$.
 \end{enumerate}
 As a byproduct, 
 an infinite family of the pairs of isomorphic Cayley graphs
 on non-isomorphic groups is obtained.

Here we show a concrete example which illustrates our 
results in this paper.
For a group $G$ and its subset $S$,
the {\em Cayley graph} ${\tt Cay}(G, S)$
is a directed graph, which has the vertex set $G$
and has an edge from $g\in G$ to $h\in G$
if and only if there exists an element $k\in S$
such that $gk=h$.
Though we call $S$ the {\em generating set} of
the graph ${\tt Cay}(G, S)$, 
we do not impose  $S$ to generate the group $G$, and hence
${\tt Cay}(G, S)$ may not be connected.
Let $\sigma=(1,3,2)$ be the cyclic permutation of
order $3$ and $\tau=(1,3)$ be the transposition
in the symmetric group ${\mathfrak S}_3$ of degree $3$.
Then each of two sets
\[
 S_1 = \{\sigma, \sigma^{-1}, \tau\},~~
 S_2 = \{\tau, \tau\sigma, {\rm id}\},
\]
generates ${\mathfrak S}_3$.
The Cayley graphs ${\tt Cay}({\mathfrak S}_3, S_1)$ and ${\tt Cay}({\mathfrak S}_3, S_2)$
are depicted in Figure $\ref{fig:pairs}$.
\begin{center}
 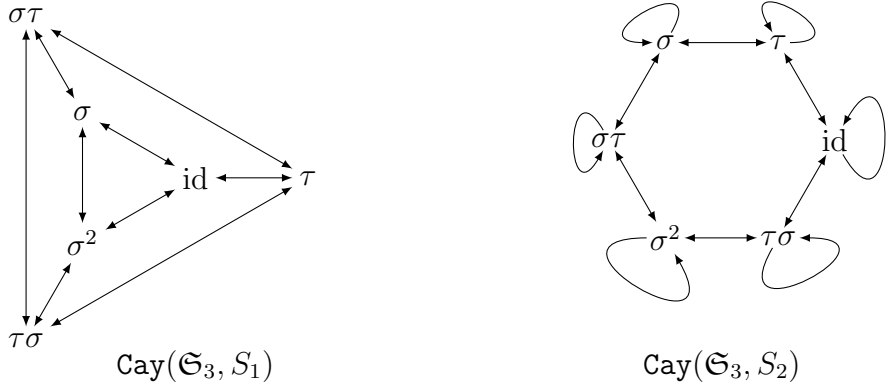
\begin{figure}[H]
  \begin{center}
   \begin{tikzpicture}
    \node[inner sep=2.5](1) at (240:1) {$\sigma^{2}$};
    \node[inner sep=2.5](2) at (120:1) {$\sigma$};
    \node[inner sep=2.5](3) at (0:1) {${\rm id}$};
    \node[inner sep=2.5](4) at (240:2.5) {$\tau\sigma$};
    \node[inner sep=2.5](5) at (120:2.5) {$\sigma\tau$};
    \node[inner sep=2.5](6) at (0:2.5) {$\tau$};
     \draw[<->,>=latex] (1)--(2);
     \draw[<->,>=latex] (2)--(3);
     \draw[<->,>=latex] (3)--(1);
     \draw[<->,>=latex] (4)--(5);
     \draw[<->,>=latex] (5)--(6);
     \draw[<->,>=latex] (6)--(4);
     \draw[<->,>=latex] (1)--(4);
     \draw[<->,>=latex] (2)--(5);
     \draw[<->,>=latex] (3)--(6);
    \draw (1,-2.5) node {${\tt Cay}({\mathfrak S}_3, S_1)$};

    \begin{scope}[xshift=8cm,yshift=0.5cm]
    \node[inner sep=1](c1) at (0:1.5) {${\rm id}$};
    \node[inner sep=1](c2) at (60:1.5) {$\tau$};
    \node[inner sep=1](c3) at (120:1.5) {$\sigma$};
    \node[inner sep=1](c4) at (180:1.5) {$\sigma\tau$};
    \node[inner sep=1](c5) at (240:1.5) {$\sigma^2$};
    \node[inner sep=1](c6) at (300:1.5) {$\tau\sigma$};
     \draw[<->,>=latex] (c1)--(c2);
     \draw[<->,>=latex] (c2)--(c3);
     \draw[<->,>=latex] (c3)--(c4);
     \draw[<->,>=latex] (c4)--(c5);
     \draw[<->,>=latex] (c5)--(c6);
     \draw[<->,>=latex] (c6)--(c1);
     \path (c1) edge[loop, looseness=10, ->, >=latex, in=60, out=-60] (c1);
     \path (c2) edge[loop, looseness=10, ->, >=latex, in=120, out=-0] (c2);
     \path (c3) edge[loop, looseness=10, ->, >=latex, in=180, out=60] (c3);
     \path (c4) edge[loop, looseness=10, ->, >=latex, in=240, out=120] (c4);
     \path (c5) edge[loop, looseness=10, ->, >=latex, in=300, out=180] (c5);
     \path (c6) edge[loop, looseness=10, ->, >=latex, in=360, out=240] (c6);
    \draw (0,-3.0) node {${\tt Cay}({\mathfrak S}_3, S_2)$};
    \end{scope}
   \end{tikzpicture}
  \end{center}
  \caption{Cayley graphs the symmetric group ${\mathfrak S}_3$ generated by $S_1$ (left) and $S_2$ (right)}
  \label{fig:pairs}
 \end{figure}
\end{center}
The diameter of ${\tt Cay}({\mathfrak S}_3, S_1)$ is $2$ and
that of ${\tt Cay}({\mathfrak S}_3, S_2)$ is $3$, hence
they are non-isomorphic.
We consider the simple random walks on these Cayley graphs. Namely
we consider the two sequences $(\mu^{(0)}_1, \mu^{(1)}_1, \mu^{(2)}_1,\ldots)$ 
and $(\mu^{(0)}_2, \mu^{(1)}_2, \mu^{(2)}_2,\ldots)$
of the probability measures over ${\mathfrak S}_3$,
which are defined as follows,
\[
 \mu_i^{(0)}(g) = \begin{cases}
		  1 & g = {\rm id},\\
		  0 & g \neq {\rm id},
		 \end{cases}
\]
\[
 \mu_i^{(t+1)}(g) = \frac{1}{3}\sum_{k\in S_i} \mu^{(t)}_i(gk^{-1}).
\]
It can be easily checked that these random walks are
ergodic and
have the uniform distribution $U$
as their stationary distributions, that is,
$\lim_{t\to\infty}\mu_i^{(t)}(g)=U(g)=1/6$ for all $g\in {\mathfrak S}_3$.
The speed of the convergence is measured by the total variation
distance ${\rm d}_{TV}(\mu_i^{(t)}, U)$ defined by
\[
 {\rm d}_{TV}(\mu_i^{(t)}, U) = 
 \frac{1}{2}\sum_{g\in {\mathfrak S}_3}\left|\mu_i^{(t)}(g)-U(g)\right|.
\]
Even though they have different diameters,
the distributions of the simple random walks on them approach the
uniform distribution in exactly the same way.
That is, 
\[
{\rm d}_{TV}(\mu_1^{(t)}, U) = {\rm d}_{TV}(\mu_2^{(t)}, U),
\]
for $ t=0,1,2,\ldots$.
This can be observed in Table $\ref{tab:mus}$, where
the values of $\mu_1^{(t)}(g)$ and $\mu_2^{(t)}(g)$
for $0\leq t\leq 5$ and $g\in {\mathfrak S}_3$ are shown.
\begin{center}
 \begin{table}[H]
  \begin{center}
\[
\begin{array}{c|cccccc}
t\backslash g & {\rm id} & \sigma & \sigma^2 & \tau & \sigma\tau & \tau\sigma \\
\hline 
0 & 1 &   0  &  0  &  0 &   0 &   0\\
1 & 0 &   1/3  &  1/3  &  1/3 &   0 &   0\\
2 & 3/9 &   1/9  &  1/9  &  0 &   2/9 &   2/9\\
3 & 2/27 &   6/27  &  6/27  &  7/27 &   3/27 &   3/27\\
4 &19/81 &  11/81  & 11/81  &  8/81 &  16/81 &  16/81\\
5 & 30/243 &  46/243 &  46/243 &  51/243 &  35/243 &  35/243
\end{array}
\]
\[
\begin{array}{c|cccccc}
t\backslash g & {\rm id} & \sigma & \sigma^2 & \tau & \sigma\tau & \tau\sigma \\
\hline 
0 &  1     &  0      & 0       &   0     & 0       &  0\\
1 &  1/3   &  0      & 0       &   1/3   & 0       &  1/3\\
2 &  3/9   &  1/9    & 1/9     &   2/9   & 0       &  2/9\\
3 &  7/27  &  3/27   & 3/27    &   6/27  & 2/27    &  6/27\\
4 & 19/81  & 11/81   &11/81    &  16/81  & 8/81    & 16/81\\
5 & 51/243 &  35/243 &  35/243 &  46/243 &  30/243 &  46/243
\end{array}
\]
  \end{center}
  \caption{Tables of $\mu_1^{(t)}$ (top) and $\mu_2^{(t)}$ (bottom)}
  \label{tab:mus}
 \end{table}
\end{center}
By comparing the two tables, we find that
$\mu_1^{(t)}$ and $\mu_2^{(t)}$ are nearly the same.
Let $\varphi_0$ and $\varphi_1$ be two permutations 
on ${\mathfrak S}_3$ defined by
\begin{equation}
 \label{eq:phi0phi1fortheta10}
 \varphi_0 =
 \begin{pmatrix}
  {\rm id} & \sigma & \sigma^2 & \tau  & \sigma\tau & \tau\sigma \\
  {\rm id} & \sigma^2 & \sigma & \sigma\tau & \tau & \tau\sigma 
 \end{pmatrix},~~~
 \varphi_1 =
 \begin{pmatrix}
  {\rm id} & \sigma & \sigma^2 & \tau  & \sigma\tau & \tau\sigma \\
  \tau & \tau\sigma & \sigma\tau & \sigma  & {\rm id} & \sigma^2
 \end{pmatrix}.
\end{equation}
Then, as we will prove later, we have
\[
 \mu_1^{(t)}(\varphi_{[t]}(g)) =
 \mu_2^{(t)}(g),
\]
where $[t]$ is an integer in $\{0,1\}$ obtained by taking modulo $2$
of $t$.
The characteristic polynomials of 
the adjacency matrices of ${\tt Cay}({\mathfrak S}_3, S_1)$
and ${\tt Cay}({\mathfrak S}_3, S_2)$ are
\[
P_1(x)=(x - 3)(x - 1)x^2(x + 2)^2,~~\mbox{ and }~~
P_2(x)=(x - 3)(x - 2)^2x^2(x + 1),
\]
from which we can observe
that 
they have common factor $P_Z(x)=(x-3)x^2$ and
\[
 \frac{P_1(x)}{P_Z(x)} = -\frac{P_2(-x)}{P_Z(-x)}.
\]
The spectral sets of the two graphs contain the
same subset $\{3,0,0\}$, whose complements in two spectral sets
$\{1,-2,-2\}$ and $\{-1,2,2\}$ have the opposite signs.

A larger example of such pairs of the Cayley graphs is
given by ${\tt Cay}(A_5, S_1)$ and ${\tt Cay}(A_5, S_2)$, where
$A_5$ is the alternating group of degree $5$ and
 \[
  S_1 = \{(1,2,3,4,5), (1,2,3,4,5)^{-1}, (1,2)(3,4)\},~~~
  S_2 = \{(1, 2)(3, 5), (1, 2)(4, 5), (1, 3)(4, 5)\}.
 \]
 See Figure $\ref{fig:anotherpair}$.
 The Cayley graph ${\tt Cay}(A_5,S_1)$ is of diameter $6$, while ${\tt Cay}(A_5, S_2)$ is
 of diameter $9$.
\begin{center}
 \begin{figure}[H]
  \begin{center}
   \begin{tikzpicture}
    \draw (0,0) node {\includegraphics[bb=0 0 587 594,clip,width=4cm]{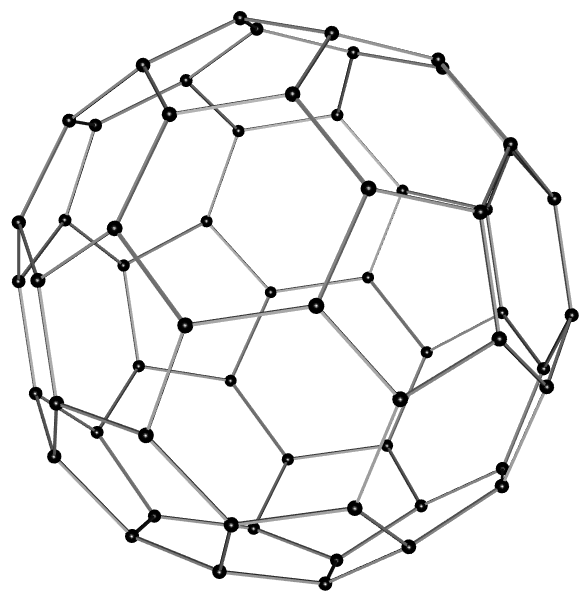}};
    \draw (0,-2.5) node {${\tt Cay}(A_5, S_1)$};
    \draw (5,0) node {\includegraphics[bb=0 0 587 594,clip,width=4cm]{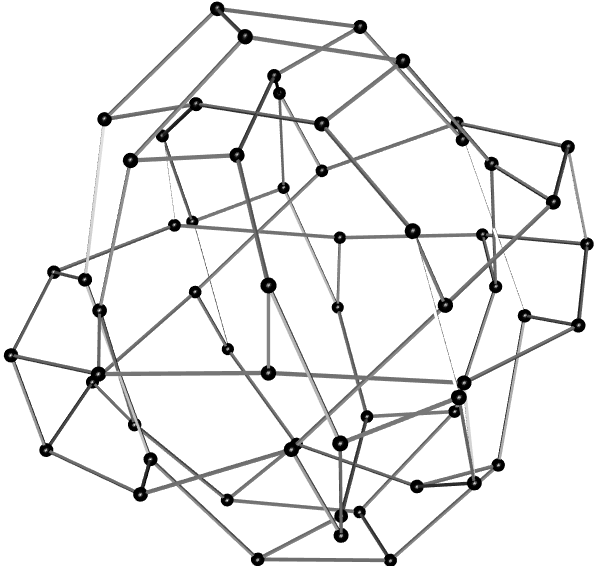}};
    \draw (5,-2.5) node {${\tt Cay}(A_5, S_2)$};
   \end{tikzpicture}
  \end{center}
  \caption{Another example of the Cayley graphs simple random walks on which have the same total variation distance
  from the uniform distribution.  All edges are bi-directed.}
  \label{fig:anotherpair}
 \end{figure}
\end{center}

Let $P_i(x)$ be the characteristic polynomial of the 
adjacency matrix of the Cayley graph ${\tt Cay}(A_5, S_i)$ for $i=1,2$.
Then $P_1(x)$ and $P_2(x)$ have the common factor
\[
 P_Z(x)=
(        x - 3)
(        x - 1)^9
(        x + 2)^4
(  x^2 - x - 3)^5
(x^2 + 3x + 1)^3,
\]
whose degree is $|A_5|/2=30$, and
\[
 \frac{P_1(x)}{P_Z(x)} =
 \frac{P_2(-x)}{P_Z(-x)}
 =
(                  x^2 + x - 4)^ 4
(                  x^2 + x - 1)^ 5
(x^4 - 3x^3 - 2x^2 + 7x + 1)^3.
\]
Thus, as multiset, the spectral set of ${\mathtt {Cay}}(A_5,S_i)$
is decomposed into a disjoint union  of two subsets $A$ and $B_i$
of the equal size $(|A|=|B_i|=|A_5|/2)$, which
satisfies $B_2 = -B_1 = \{-\lambda\,|\,\lambda\in B_1\}$.

Our main aim in this paper is to 
construct an infinite family of the
triples $(G,S_1, S_2)$ consisting of a group $G$ and a pair 
$(S_1, S_2)$ of distinct subsets of $G$
with the  properties ${\bf \ref{cond:nonisomorphic}}, {\bf \ref{cond:RW}},$ and ${\bf \ref{cond:spectral}}$.
Our construction uses the sliding block puzzles
defined on a certain family of graphs, which can be considered
as a generalization of the one studied in $\cite{hanaoka20235}$.

The outline of the paper is as follows.
In Section $2$, we show a construction of an
infinite family of triples.
In Section $3$, we show the triples constructed
in Section $2$ have the property ${\bf \ref{cond:RW}}$.
In Section $4$, by applying the theory of the zeta
functions of finite graph covering, 
we show the triples have the properties ${\bf \ref{cond:nonisomorphic}}$
and ${\bf \ref{cond:spectral}}$.

\section{Construction of pairs}
\label{sec:construction}

Let $\Gamma$ be a finite undirected simple graph equipped with the vertex set $V$ 
of cardinality $n+1$ and 
the edge set $E$. 
A {\em position} of the {\em sliding block puzzle}
defined on $\Gamma$ is a bijection $f: V\to \{1,2,\ldots, n, 0\}$.
We say $v\in V$ is {\em blank} or {\em unoccupied} in
a position $f$ if $f(v)=0$, and hence, 
the vertex $f^{-1}(0)$ is blank.
A position $f$ is transformed into another position $g$  by a
{\em move} if there exist two mutually
adjacent vertices $v,w$ in $\Gamma$, such that, $v$ or $w$
is blank in $f$, and
\[
 g = f\circ (v,w),
\]
where $(v,w)$ denotes the transposition of $v$ and $w$.
Since $\Gamma$ is undirected and simple,
every path $p$ on $\Gamma$ can be represented as a sequence
$(v_0,v_1,\ldots,v_l)$ of the vertices. 
We define the permutation
$\sigma_p$ of the vertices by
\[
 \sigma_p = (v_0,v_1)(v_1,v_2)\cdots (v_{l-1}, v_{l}).
\]
If $g=f\circ\sigma_p$ and $v_0$ is blank in $f$, then
$v_l$ is blank in $g$.
That is, $\sigma_p$ sends the blank from $v_0$ to $v_l$
through the path $p$.

Let ${\rm puz}(\Gamma)$ be the graph whose vertex set consists of
the positions of the puzzle.
Two vertices (or positions) $f$ and $g$ of ${\rm puz}(\Gamma)$
are connected by an edge $\{f,g\}$
if and only if there is a move transforming $f$ into $g$.
\begin{center}
 \begin{figure}[H]
  \begin{center}
   \begin{tikzpicture}
    \draw (0:1) -- (60:1) -- (120:1) -- (180:1) -- (240:1) -- (300:1) -- cycle;
    \draw (0:1) -- (180:1);
    \draw[fill=black, stroke=black] (0,0) circle (0.05);
    \draw[fill=black, stroke=black] (0:1) circle (0.05);
    \draw[fill=black, stroke=black] (60:1) circle (0.05);
    \draw[fill=black, stroke=black] (120:1) circle (0.05);
    \draw[fill=black, stroke=black] (180:1) circle (0.05);
    \draw[fill=black, stroke=black] (240:1) circle (0.05);
    \draw[fill=black, stroke=black] (300:1) circle (0.05);
   \end{tikzpicture}
   \caption{The graph $\theta_0$}\label{fig:thetazero}
  \end{center}
 \end{figure}
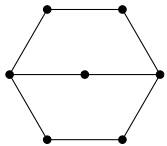
\end{center}
Wilson \cite{wilson1974graph} shows the following fundamental theorem for the sliding
block puzzles:
\begin{thm}\label{thm:wilson}{\rm (\cite[Theorem 1]{wilson1974graph})}
Let $\Gamma$ be a finite simple connected graph other
than a polygon or the graph $\theta_0$ shown in Figure $\ref{fig:thetazero}$.
Then ${\rm puz}(\Gamma)$ is connected unless $\Gamma$ is bipartite,
in which case ${\rm puz}(\Gamma)$ has exactly two components.
In this latter case, positions $f,g$ on $\Gamma$ having
blank vertices at even (resp. odd) distance in $\Gamma$
are in the same component of ${\rm puz}(\Gamma)$ if and only if
there exists an even (resp. odd) permutation $\sigma$
of $V$ such that $f\circ\sigma = g$.
\end{thm}

\begin{dfn}
Let $X$ be an undirected simple graph.
Let $v$ be a vertex of degree two,
$\{v_1, v_2\}$ be the neighboring vertex of $v$,
and
$e_1$ and $e_2$ be the two edges incident to $v$.
By replacing $e_1$ and $e_2$
with a single edge $e$ connecting $v_1$ and $v_2$
vertices neighboring $v$
and removing $v$ from $X$,
we obtain a smaller graph.
By repeatedly applying this procedure,
we obtain a graph $\widetilde{X}$
without vertices of degree two.
We call $\widetilde{X}$ the {\em path contraction}
of $X$.
\end{dfn}

For a positive integer $a$ and
a non-negative integer $b$, 
the {\em theta graph} $\theta_{a,b}$ 
is defined as follows.
The theta graph $\theta_{a,b}$ has the vertex set $V=\{v_0, v_1, \ldots, v_n\}$
where $n=2a+b+1$,
and the edge set $E$ defined by
\[
 E  = \left\{\{v_i, v_{i+1}\}\,\middle|\, i\in\{0,1,\ldots,n-2\}\backslash\{2a+1\}
 \right\}
 \cup 
 \left\{
 \{v_0, v_{2a+1}\},
 \{v_0,v_{2a+2}\},
 \{v_{2a+b+1}, v_{a+1}\}
 \right\},
\]
where $\{v,w\}$ denotes the unique edge
connecting $v$ and $w$ in $\theta_{a,b}$.
Figure $\ref{fig:theta23}$ shows $\theta_{2,3}$.
\begin{center}
 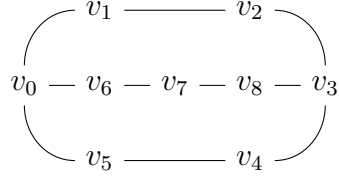
\begin{figure}[H]
  \begin{center}
   \begin{tikzpicture}
    \node (v0) at (0,0) {$v_0$};
    \node (v6) at (1,0) {$v_6$};
    \node (v7) at (2,0) {$v_7$};
    \node (v8) at (3,0) {$v_8$};
    \node (v3) at (4,0) {$v_3$};
    \node (v1) at (1.,1) {$v_1$};
    \node (v2) at (3,1) {$v_2$};
    \node (v5) at (1.,-1) {$v_5$};
    \node (v4) at (3,-1) {$v_4$};
    \draw[out=90,in=180] (v0) to (v1);
    \draw (v1) to (v2);
    \draw[out=0,in=90] (v2) to (v3);
    \draw[out=-90,in=0] (v3) to (v4);
    \draw (v4) to (v5);
    \draw[out=180,in=-90] (v5) to (v0);
    \draw (v0) to (v6);
    \draw (v6) to (v7);
    \draw (v7) to (v8);
    \draw (v8) to (v3);
   \end{tikzpicture}
   \caption{$\theta_{2,3}$}
   \label{fig:theta23}
  \end{center}
 \end{figure}
\end{center}
Let $\rho: V\to V$ be the bijection
defined by
\[
 \rho(v_i) = \begin{cases}
		v_{(i+a+1){\rm mod }(2a+2)} & i = 0,1,\ldots, 2a+1,\\
		v_{4a+b+3-i} & i = 2a+2, 2a+3, \ldots, 2a+b+1.
	       \end{cases}
\]
Then $\rho$ induces a graph automorphism
of $\theta_{a,b}$, which can be considered as the
$180^\circ$ rotation.
Let $\psi: V\to V$ be the bijection defined by
\[
 \psi(v_i) = \begin{cases}
	      v_{a+1-i} & i = 0,1,\ldots, a+1,\\
	      v_{3a+3-i} & i = a+2,a+3,\ldots,2a+1,\\
	      v_{4a+b+3-i} & i = 2a+2, 2a+3, \ldots, 2a+b+1.
	     \end{cases}
\]
Then $\psi$ induces another graph automorphism of $\theta_{a,b}$, which can be considered as a
vertical flip.
These two automorphisms $\rho$ and $\psi$ generate
a subgroup of the automorphism group ${\rm Aut}(\theta_{a,b})$ of $\theta_{a,b}$,
which is isomorphic to the Klein four-group.
The image of each vertex $v_i$ in $\theta_{2,3}$ of $\rho$ (resp. $\psi$)
are shown in the left (resp. right) side of Figure $\ref{fig:rhopsi}$.

\begin{center}
 \begin{figure}[H]
  \begin{center}
   \begin{tikzpicture}
    \node (v0) at (0,0) {$v_3$};
    \node (v6) at (1,0) {$v_8$};
    \node (v7) at (2,0) {$v_7$};
    \node (v8) at (3,0) {$v_6$};
    \node (v3) at (4,0) {$v_0$};
    \node (v1) at (1.,1) {$v_2$};
    \node (v2) at (3,1) {$v_1$};
    \node (v5) at (1.,-1) {$v_4$};
    \node (v4) at (3,-1) {$v_5$};
    \draw[out=90,in=180] (v0) to (v1);
    \draw (v1) to (v2);
    \draw[out=0,in=90] (v2) to (v3);
    \draw[out=-90,in=0] (v3) to (v4);
    \draw (v4) to (v5);
    \draw[out=180,in=-90] (v5) to (v0);
    \draw (v0) to (v6);
    \draw (v6) to (v7);
    \draw (v7) to (v8);
    \draw (v8) to (v3);
    \begin{scope}[xshift=-6cm]
    \node (v0) at (0,0) {$v_3$};
    \node (v6) at (1,0) {$v_8$};
    \node (v7) at (2,0) {$v_7$};
    \node (v8) at (3,0) {$v_6$};
    \node (v3) at (4,0) {$v_0$};
    \node (v1) at (1.,1) {$v_4$};
    \node (v2) at (3,1) {$v_5$};
    \node (v5) at (1.,-1) {$v_2$};
    \node (v4) at (3,-1) {$v_1$};
    \draw[out=90,in=180] (v0) to (v1);
    \draw (v1) to (v2);
    \draw[out=0,in=90] (v2) to (v3);
    \draw[out=-90,in=0] (v3) to (v4);
    \draw (v4) to (v5);
    \draw[out=180,in=-90] (v5) to (v0);
    \draw (v0) to (v6);
    \draw (v6) to (v7);
    \draw (v7) to (v8);
    \draw (v8) to (v3);
    \end{scope}
   \end{tikzpicture}
   \caption{Images of $\rho$ (left) and $\psi$ (right) for $\theta_{2,3}$}
   \label{fig:rhopsi}
  \end{center}
 \end{figure}
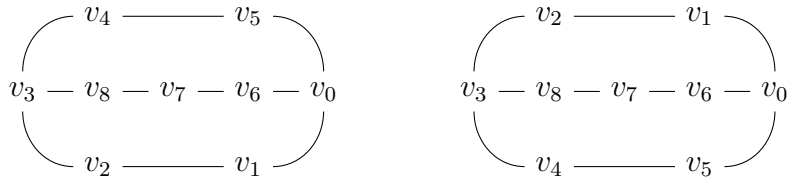
\end{center}

There exist three simple paths $p_1, p_2, p_3$
from $v_0$ and $v_{a+1}$, defined by
\[
 p_1 = (v_0, v_1, \ldots, v_{a+1}),~~
 p_2 = (v_0, v_{2a+2}, \ldots, v_{2a+b+1}, v_{a+1}),~~
 p_3 = (v_0, v_{2a+1}, \ldots, v_{a+1}).
\]
Each path $p_k$ defines a vertex permutation $\sigma_{p_k}$ of $\theta_{a,b}$.
The vertex permutation $\sigma_{p_k}\rho$ fixes the
vertex $v_0$, which gives the permutation $\sigma_k \in {\mathfrak S}_n$
such that
\begin{equation}\label{eq:defsigmak}
 \sigma_{p_k}\rho(v_i)=v_{\sigma_k(i)}, 
\end{equation}
for $i=1,2,\ldots,n$.
We define $S_1$ by
\begin{equation}
\label{eq:S1}
 S_1  = \left\{\sigma_1,\sigma_2,\sigma_3\right\}.
\end{equation}
In the same manner, we define $\tau_k\in {\mathfrak S}_3$
by using $\psi$ instead of $\rho$, that is,
\[
 \sigma_{p_k}\psi(v_i)=v_{\tau_k(i)},
\]
for $k=1,2,3$, and we define $S_2$ by
\begin{equation}\label{eq:S2}
 S_2 = \left\{\tau_1,\tau_2,\tau_3\right\}.
\end{equation}

Let $S$ be a subset of a finite group $G$.
The {\em Cayley graph} ${\tt Cay}(G,S)$ of the group $G$
with the generating set $S$ is an directed graph,
whose vertex set is $G$ and there exists an edge from $g$ to $h$ in $G$
if and only if there exists an element $k\in S$ such that $h = gk$.
For an edge $e$ of ${\tt Cay}(G,S)$, the starting vertex of $e$
is denoted $o(e)$, and the terminal vertex is denoted $t(e)$.

We then consider
the Cayley graph of the group $G\times C_2$,
where
$C_2={\mathbb Z}/2{\mathbb Z}=\{0,1\}$ is
the cyclic group of order $2$.
For the sake of simplicity of the notation,
we write
\[
 X(G,S)={\tt Cay}(G,S),~~~Y(G,S)={\tt Cay}(G\times C_2, S\times\{1\}).
\]
Then the Cayley graph
$Y(G, S)$ is a bipartite graph, that is,
the vertex set $G\times C_2$ of $Y(G,S)$
is divided into two disjoint
parts $B=G\times\{0\}$ and $W=G\times \{1\}$
and every edge is from $B$ to $W$ or
$W$ to $B$.

Let $f_0$ be the {\em initial} position
which is defined by $f_0(v_i)=i$
for $i\in\{0,1,\ldots,n\}$.
We are interested in the connected component
${\rm puz}_0(\theta_{a,b})$ of 
${\rm puz}(\theta_{a,b})$, which contains
the initial position $f_0$.
It is clear that the maps
\[
 f\mapsto f\rho,~~~
 f\mapsto f\psi,
\]
can be considered as two automorphisms of $\widetilde{\rm puz}(\theta_{a,b})$,
since they are the $180^\circ$ rotation and the vertical flip
of $\theta_{a,b}$ respectively.
Thus, we can consider the group
$K=\{{\rm id}, \rho, \psi, \rho\psi\}$
acts also on $\widetilde{\rm puz}(\theta_{a,b})$ as
a subgroup of the automorphisms group ${\rm Aut}(\widetilde{\rm puz}(\theta_{a,b}))$.
When $a+b$ is an even integer, by Theorem $\ref{thm:wilson}$,
${\rm puz}(\theta_{a,b})$ has two connected components.
The following lemma shows the conditions
for each of $\rho$ and $\psi$ to be 
contained in ${\rm Aut}(\widetilde{\rm puz}_0(\theta_{a,b}))$.

\begin{lem}
 \label{lem:rhopsiAuto}
 If $a\equiv 0\pmod{2}$ and $b\equiv 0\pmod{4}$, then
 \[
 \rho, \psi \in
  {\rm Aut}\left(\widetilde{\rm puz}_0(\theta_{a,b})\right),
 \mbox{~~ and ~~~}
 S_1, S_2 \subset A_n.
 \]
 If $a\equiv 0\pmod{2}$ and $b\equiv 2\pmod{4}$, then
 \[
 \rho, \psi \not\in
  {\rm Aut}\left(\widetilde{\rm puz}_0(\theta_{a,b})\right),
 \mbox{~~ and ~~~}
 S_1, S_2 \subset {\mathfrak S}_n\backslash A_n.
 \]
 If $a\equiv 1\pmod{2}$ and $b\equiv 1\pmod{4}$, then
 \[
 \rho\in
  {\rm Aut}\left(\widetilde{\rm puz}_0(\theta_{a,b})\right),~~
 \psi\not\in
  {\rm Aut}\left(\widetilde{\rm puz}_0(\theta_{a,b})\right),~~
 S_1 \subset A_n,
 \mbox{~ and ~}
 S_2 \subset {\mathfrak S}_n\backslash A_n.
 \]
 If $a\equiv 1\pmod{2}$ and $b\equiv 3\pmod{4}$, then
 \[
 \rho\not\in
  {\rm Aut}\left(\widetilde{\rm puz}_0(\theta_{a,b})\right),~~
 \psi\in
  {\rm Aut}\left(\widetilde{\rm puz}_0(\theta_{a,b})\right),~~
 S_1 \subset {\mathfrak S}_n\backslash A_n,
 \mbox{~ and ~}
 S_2 \subset A_n.
 \]
\end{lem}
\begin{proof}
 First note that $\rho$ can be expressed as
 \[
 \rho = \prod_{i=0}^a (v_i, v_{a+1+i})
 \prod_{j=1}^{\lfloor b/2\rfloor} (v_{2a+1+j}, v_{n+1-j})
 \]
 which is the product of $a+1+\lfloor b/2\rfloor$ disjoint
 transpositions of the vertices of $\theta_{a,b}$,
 and $\psi$ can be expressed as
 \[
  \psi = (v_0, v_{a+1})\prod_{i=1}^{\lfloor a/2\rfloor}
 (v_i, v_{a+1-i})(v_{a+1+i},v_{2a+2-i})
 \prod_{j=1}^{\lfloor b/2\rfloor} (v_{2a+1+j}, v_{n+1-j})
 \]
 which is the product of an even number of disjoint
 transpositions if and only if $\lfloor b/2\rfloor$
 is an odd integer.
 If $a\equiv 0\pmod{2}$ and $b\equiv 0\pmod{4}$, then
 the distance from $v_0$ to $v_{a+1}$ is odd and
 $\rho$ can be expressed as the product of 
 an odd number of vertex transpositions, and
 $f\rho$ is a vertex of $\widetilde{\rm puz}_0(\theta_{a,b})$
 for every vertex $f$ of $\widetilde{\rm puz}_0(\theta_{a,b})$.
 Since $\psi$ is also the product of the odd number
 ($a+1+b/2$) of disjoint transpositions of the vertices,
 and $f\psi$
 is also a vertex of $\widetilde{\rm puz}_0(\theta_{a,b})$
 for every vertex $f$ of $\widetilde{\rm puz}_0(\theta_{a,b})$.
 Since $\sigma_{p_k}$ is an odd permutation of the vertices
 for $k=1,2,3$,
 $\sigma_{p_k}\rho$ and $\sigma_{p_k}\psi$ are  both even permutations fixing $v_0$.
 Thus we have $S_1, S_2\subset A_n$.
 Another cases can be shown in the same manner.
\end{proof}

In the following theorem, we consider the graph $\widetilde{\rm puz}(\theta_{a,b})$ as
a bi-directed graph, that is, if there is an edge
$\{f,g\}$ connecting $f$ and $g$ in $\widetilde{\rm puz}(\theta_{a,b})$, 
we consider there are two edges of opposite directions, from $f$ to $g$ and from $g$ to $f$.
Therefore, the symbol $\cong$ in the theorem stands for the isomorphism
of two directed graphs.
\begin{thm}
 \label{thm:existence}
 Let $a$ be a positive integer, and
 $b$ a non-negative integer such that
 $(a,b)\neq (2,1)$.
 Let $S_1$ and $S_2$
 be defined by $(\ref{eq:S1})$ and $(\ref{eq:S2})$ respectively. 
 Then, we have
 \begin{equation}
  \label{eq:isomorphism1}
 \widetilde{{\rm puz}}(\theta_{a,b})
 \cong
 Y({\mathfrak S}_n, S_1)
 \cong
 Y({\mathfrak S}_n, S_2),
 \end{equation}
 where $n=2a+b+1$ and $\widetilde{{\rm puz}}(\theta_{a,b})$
 is the path contraction of ${\rm puz}(\theta_{a,b})$.

 If $a+b$ is an odd integer, then each of
 $S_1$ and $S_2$ generates the symmetric group ${\mathfrak S}_n$ 
 of degree $n$, and $\widetilde{\rm puz}(\theta_{a,b})$ is a
 strongly connected graph.

 If $a\equiv 0\pmod{2}$ and $b\equiv 0\pmod{4}$, then
 \[
  \widetilde{\rm puz}_0(\theta_{a,b})
 \cong Y(A_n, S_1) \cong Y(A_n, S_2),
 \]
 and each of $S_1$ and $S_2$ generates $A_n$.

 If $a\equiv 0\pmod{2}$ and $b\equiv 2\pmod{4}$, then
 \[
  \widetilde{\rm puz}_0(\theta_{a,b})
 \cong X({\mathfrak S}_n, S_1) \cong X({\mathfrak S}_n, S_2),
 \]
 and each of $S_1$ and $S_2$ generates ${\mathfrak S}_n$.

 If $a\equiv 1\pmod{2}$ and $b\equiv 1\pmod{4}$, then
 \[
  \widetilde{\rm puz}_0(\theta_{a,b})
 \cong Y(A_n, S_1) \cong  X({\mathfrak S}_n, S_2),
 \]
 $S_1$ generates $A_n$, and $S_2$ generates ${\mathfrak S}_n$.

 If $a\equiv 1\pmod{2}$ and $b\equiv 3\pmod{4}$, then
 \[
  \widetilde{\rm puz}_0(\theta_{a,b})
 \cong X({\mathfrak S}_n, S_1) \cong  Y(A_n, S_2),
 \]
 $S_1$ generates ${\mathfrak S}_n$, and $S_2$ generates $A_n$.
\end{thm}

\begin{proof}
Given a position $f$, we define the permutation $\sigma_f\in {\mathfrak S}_n$
as follows.
If $f$ has the blank at $v_0$, that is $f(v_0)=0$, then
$\sigma_f$ is defined by 
\[
 f(v_i) = \sigma_f(i),
\]
for $i=1,2,\ldots,n$.
If $f$ has the blank at $v_{a+1}$, then
$f\rho(v_0)=0$ and 
$\sigma_f$  is the unique permutation which satisfies
\[
 f\rho(v_i) = \sigma_f(i),
\]
for $i=1,2,\ldots,n$.
Let $c_f\in \{0,1\}$ be defined by
\[
 c_f = \begin{cases}
	0 & \mbox{~~ if ~~} f(v_0) = 0,\\
	1 & \mbox{~~ if ~~} f(v_{a+1}) = 0.\\
       \end{cases}
\]
Then the map
\[
 f \mapsto (\sigma_f, c_f)
\]
is a bijection from the vertex set of $\widetilde{{\rm puz}}(\theta_{a,b})$
to ${\mathfrak S}_n\times C_2$.
Let $f$ be a vertex of $\widetilde{{\rm puz}}(\theta_{a,b})$.
Then, since  $\widetilde{{\rm puz}}(\theta_{a,b})$ is obtained
by the path contraction,
$f^{-1}(0)\in\{v_0, v_{a+1}\}$.
If $f^{-1}(0)=v_0$ and there exists an edge from $f$ to a position $g$,
there exists $k\in\{1,2,3\}$  such that
$g=f \sigma_{p_k}$ and $g(v_{a+1})=0$. 
Therefore, using the relation $(\ref{eq:defsigmak})$ we obtain
$ g\rho(v_i) = f\sigma_{p_k}\rho(v_i) = f(v_{\sigma_k(i)})$.
This implies
\begin{equation}
 \label{eq:cayley}
 \sigma_g = \sigma_f\sigma_{k}. 
\end{equation}
If $f(v_{a+1})=0$ and there exists an edge from $f$ to a position $g$,
there exists $k\in\{1,2,3\}$ such that
$g=f \rho\sigma_{p_k}\rho$ and $g(v_{0})=0$. Then,
\[
 \sigma_g(i) = g(v_i) = f\rho\sigma_{p_k}\rho(v_i)
 = f\rho(v_{\sigma_k(i)}) = \sigma_f(\sigma_k(i)),
\]
and $(\ref{eq:cayley})$ holds.
Thus we have obtained the first isomorphism in $(\ref{eq:isomorphism1})$.
The second isomorphism in $(\ref{eq:isomorphism1})$ can be
obtained by using $\psi$ instead of $\rho$.

Let ${V}_{a,b}$ be the vertex set of ${\rm puz}(\theta_{a,b})$ and
$\widetilde{V}_{a,b}$ the vertex set of $\widetilde{\rm puz}(\theta_{a,b})$.
It is clear that $|{V}_{a,b}| = (2a+b+2)!$, since
$\theta_{a,b}$ has $2a+b+2$ vertices.
Since a vertex in $\widetilde{V}_{a,b}$ can be identified with
a vertex in $V_{a,b}$ of degree three, we have
\[
 |V_{a,b}|=|\widetilde{V}_{a,b}|+\frac{1}{2}\left(2a|\widetilde{V}_{a,b}| + b|\widetilde{V}_{a,b}|\right)
 =\frac{2a+b+2}{2}|\widetilde{V}_{a,b}|.
\]
Thus we have
\[
 |\widetilde{V}_{a,b}| = 2 n!.
\]

If $a+b$ is an odd integer, then 
${\rm puz}(\theta_{a,b})$ is a connected graph.
Therefore both  $S_1\times\{1\}$ and $S_2\times\{1\}$
generate the group ${\mathfrak S}_n\times C_2$,
and hence
each of $S_1$ and $S_2$ generates ${\mathfrak S}_n$.

If $a+b$ is an even integer, then $\theta_{a,b}$ is bipartite
and ${\rm puz}(\theta_{a,b})$ has two isomorphic connected components,
${\rm puz}_0(\theta_{a,b})$ and the other one.
Let $\widetilde{V}^0_{a,b}$ be the vertex set of $\widetilde{\rm puz}_0(\theta_{a,b})$.
Then it is clear that
\[
 \left|\widetilde{V}^0_{a,b}\right| = \frac{1}{2} \left|\widetilde{V}_{a,b}\right|
 = n!.
\]

If $a\equiv 0\pmod{2}$ and $b\equiv 0\pmod{4}$, then,
by Lemma $\ref{lem:rhopsiAuto}$,
we have $S_1, S_2\subset A_n$.
Since $Y({\mathfrak S}_n, S_i)$ has two connected
components of size $n!$, one of the two components
has the vertex set $A_n\times C_2$,
and the other $\left({\mathfrak S}_n\backslash A_n\right)\times C_2$
 for $i=1,2$.
Thus each of two components is isomorphic to
$Y(A_n, S_i)$,
which implies that each of $S_1$ and $S_2$ generates $A_n$.

If $a\equiv 0\pmod{2}$ and $b\equiv 2\pmod{4}$, then,
by Lemma $\ref{lem:rhopsiAuto}$,
we have $S_1, S_2\subset {\mathfrak S}_n\backslash A_n$.
One of the two connected components of
$Y({\mathfrak S}_n, S_i)$ contains the vertex $({\rm id},0)$,
and the other contains $({\rm id},1)$.
Each component is isomorphic to $X({\mathfrak S}_n, S_i)$
for $i=1,2$,
which implies each of $S_1$ and $S_2$ generates ${\mathfrak S}_n$.

If $a\equiv 1\pmod{2}$ and $b\equiv 1\pmod{4}$, then,
by Lemma $\ref{lem:rhopsiAuto}$,
we have $S_1\subset A_n$ and  $S_2\subset {\mathfrak S}_n\backslash A_n$.
Hence $Y({\mathfrak S}_n, S_1)$ has two connected components,
each of which is isomorphic to $Y(A_n, S_1)$, and
$Y({\mathfrak S}_n, S_2)$ has two connected components,
each of which is isomorphic to $X({\mathfrak S}_n, S_2)$.
This implies $S_1$ generates $A_n$
and $S_2$ generates ${\mathfrak S}_n$.

If $a\equiv 1\pmod{2}$ and $b\equiv 3\pmod{4}$, then,
by Lemma $\ref{lem:rhopsiAuto}$,
we have $S_1\subset {\mathfrak S}_n\backslash A_n$ and  $S_2\subset A_n$,
and the statement for this case can be proved
in the same manner as the previous case.
\end{proof}

\begin{rmk}
 Theorem $\ref{thm:existence}$ gives examples
 of the pairs of isomorphic Cayley graphs on the non-isomorphic groups ${\mathfrak S}_n$ and $A_n\times C_2$.
 For instance, when $a=b=1$, we have
 \[
  Y(A_4, \left\{(1,3,2),(1,3)(2,4),(1,2,3)\right\})
 \cong
 X({\mathfrak S}_4, \left\{(1,2),(2,4),(2,3)\right\}).
 \]
\end{rmk}

\begin{exm}\label{exm:theta1b}
 Let $a=1$ and $b=2m$ for non-negative integer $m$.
 Then $\widetilde{\rm puz}(\theta_{a,b})$ is connected
 and we have
 \begin{eqnarray*}
  \sigma_1 & = &
 \begin{cases}
 \begin{pmatrix}
  1 & 2 & 3 & 4 & 5 & \cdots & n\\
  3 & 1 & 2 & n & n-1& \cdots & 4
 \end{pmatrix}
  & m >0, \\\\
  (1,3,2) & m = 0.
 \end{cases}
 \\
  \sigma_2 & = &
 \begin{cases}
  \begin{pmatrix}
   1 & 2 & 3 & 4 & 5 & \cdots & n\\
   3 & 4 & 1 & 2 & n & \cdots & 5
  \end{pmatrix} &
  m > 0,\\\\
  (1,3) &
  m = 0.
 \end{cases}
 \\
  \\
  \sigma_3 & = & \sigma_1^{-1}.
\end{eqnarray*}
\begin{eqnarray*}
  \tau_1 & = &(2,3)\sigma_1 = 
 \begin{pmatrix}
  1 & 2 & 3 & 4 & 5 & \cdots & n\\
  2 & 1 & 3 & n & n-1& \cdots & 4
 \end{pmatrix}.
 \\
 \tau_2 & = & (1,3)\sigma_2 =
 \begin{cases}
 \begin{pmatrix}
  1 & 2 & 3 & 4 & 5 & \cdots & n\\
  1 & 4 & 3 & 2 & n & \cdots & 5
 \end{pmatrix} &
  m > 0,\\\\
  {\rm id} & m = 0.
 \end{cases}
 \\
  \tau_3 & = & (1,2)\sigma_3 =
 \begin{pmatrix}
  1 & 2 & 3 & 4 & 5 & \cdots & n\\
  1 & 3 & 2 & n & n-1& \cdots & 4
 \end{pmatrix} .
\end{eqnarray*}
 The diameters of the Cayley graphs $X({\mathfrak S}_{2m+3}, S_i)$ for $i=1,2$ and $m=0,1,2$
 are listed in the following table.
 \[
  \begin{array}{c|ccc}
   i\backslash m & 0 & 1 & 2\\
   \hline
   1 & 2 & 10 & 17\\
   2 & 3 & 9 & 17
  \end{array}
 \]
 The first example shown in Section $1$ is obtained
 from the case when $m=0$. (See Figure $\ref{fig:XY}$.)
 \begin{center}
  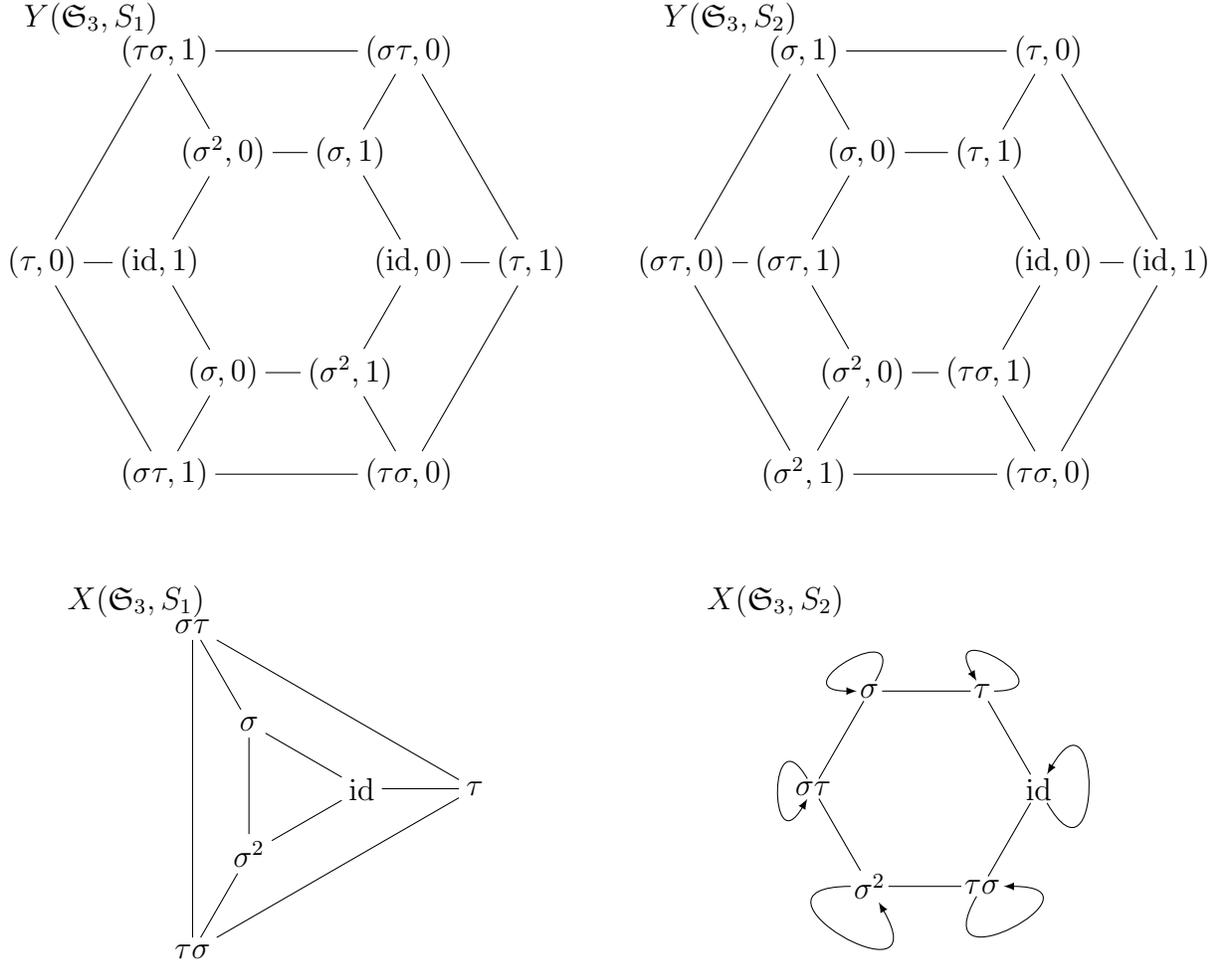
\begin{figure}[H]
   \begin{center}
    \begin{tikzpicture}[scale=1]
    \node[inner sep=2.5](1) at (240:1) {$\sigma^{2}$};
    \node[inner sep=2.5](2) at (120:1) {$\sigma$};
    \node[inner sep=2.5](3) at (0:1) {${\rm id}$};
    \node[inner sep=2.5](4) at (240:2.5) {$\tau\sigma$};
    \node[inner sep=2.5](5) at (120:2.5) {$\sigma\tau$};
    \node[inner sep=2.5](6) at (0:2.5) {$\tau$};
     \draw[>=latex] (1)--(2);
     \draw[>=latex] (2)--(3);
     \draw[>=latex] (3)--(1);
     \draw[>=latex] (4)--(5);
     \draw[>=latex] (5)--(6);
     \draw[>=latex] (6)--(4);
     \draw[>=latex] (1)--(4);
     \draw[>=latex] (2)--(5);
     \draw[>=latex] (3)--(6);
    \draw (-2,2.5) node {$X({\mathfrak S}_3, S_1)$};
     \begin{scope}[yshift=7cm,scale=1.3]
    \node[inner sep=2.5](10) at (120:1.3) {$(\sigma^{2},0)$};
    \node[inner sep=2.5](21) at (60:1.3) {$(\sigma,1)$};
    \node[inner sep=2.5](30) at (0:1.3) {$({\rm id},0)$};
    \node[inner sep=2.5](11) at (300:1.3) {$(\sigma^{2},1)$};
    \node[inner sep=2.5](20) at (240:1.3) {$(\sigma,0)$};
    \node[inner sep=2.5](31) at (180:1.3) {$({\rm id},1)$};
    \node[inner sep=2.5](41) at (120:2.5) {$(\tau\sigma,1)$};
    \node[inner sep=2.5](50) at (60:2.5) {$(\sigma\tau,0)$};
    \node[inner sep=2.5](61) at (0:2.5) {$(\tau,1)$};
    \node[inner sep=2.5](40) at (300:2.5) {$(\tau\sigma,0)$};
    \node[inner sep=2.5](51) at (240:2.5) {$(\sigma\tau,1)$};
    \node[inner sep=2.5](60) at (180:2.5) {$(\tau,0)$};
     \draw[>=latex] (10)--(21);
     \draw[>=latex] (21)--(30);
     \draw[>=latex] (30)--(11);
     \draw[>=latex] (11)--(20);
     \draw[>=latex] (20)--(31);
     \draw[>=latex] (31)--(10);
     \draw[>=latex] (41)--(50);
     \draw[>=latex] (50)--(61);
     \draw[>=latex] (61)--(40);
     \draw[>=latex] (40)--(51);
     \draw[>=latex] (51)--(60);
     \draw[>=latex] (60)--(41);
     \draw[>=latex] (10)--(41);
     \draw[>=latex] (21)--(50);
     \draw[>=latex] (30)--(61);
     \draw[>=latex] (11)--(40);
     \draw[>=latex] (20)--(51);
     \draw[>=latex] (31)--(60);
      \draw (-2,2.5) node {$Y({\mathfrak S}_3, S_1)$};
     \end{scope}
     \begin{scope}[xshift=8.5cm]
    \node[inner sep=1](c1) at (0:1.5) {${\rm id}$};
    \node[inner sep=1](c2) at (60:1.5) {$\tau$};
    \node[inner sep=1](c3) at (120:1.5) {$\sigma$};
    \node[inner sep=1](c4) at (180:1.5) {$\sigma\tau$};
    \node[inner sep=1](c5) at (240:1.5) {$\sigma^2$};
    \node[inner sep=1](c6) at (300:1.5) {$\tau\sigma$};
     \draw[>=latex] (c1)--(c2);
     \draw[>=latex] (c2)--(c3);
     \draw[>=latex] (c3)--(c4);
     \draw[>=latex] (c4)--(c5);
     \draw[>=latex] (c5)--(c6);
     \draw[>=latex] (c6)--(c1);
     \path (c1) edge[loop, looseness=10, ->, >=latex, in=60, out=-60] (c1);
     \path (c2) edge[loop, looseness=10, ->, >=latex, in=120, out=-0] (c2);
     \path (c3) edge[loop, looseness=10, ->, >=latex, in=180, out=60] (c3);
     \path (c4) edge[loop, looseness=10, ->, >=latex, in=240, out=120] (c4);
     \path (c5) edge[loop, looseness=10, ->, >=latex, in=300, out=180] (c5);
     \path (c6) edge[loop, looseness=10, ->, >=latex, in=360, out=240] (c6);
    \draw (-2,2.5) node {$X({\mathfrak S}_3, S_2)$};
     \begin{scope}[yshift=7cm,scale=1.3]
    \node[inner sep=2.5](10) at (120:1.3) {$(\sigma,0)$};
    \node[inner sep=2.5](21) at (60:1.3) {$(\tau,1)$};
    \node[inner sep=2.5](30) at (0:1.3) {$({\rm id},0)$};
    \node[inner sep=2.5](11) at (300:1.3) {$(\tau\sigma,1)$};
    \node[inner sep=2.5](20) at (240:1.3) {$(\sigma^{2},0)$};
    \node[inner sep=2.5](31) at (180:1.3) {$(\sigma\tau,1)$};
    \node[inner sep=2.5](41) at (120:2.5) {$(\sigma,1)$};
    \node[inner sep=2.5](50) at (60:2.5) {$(\tau,0)$};
    \node[inner sep=2.5](61) at (0:2.5) {$({\rm id},1)$};
    \node[inner sep=2.5](40) at (300:2.5) {$(\tau\sigma,0)$};
    \node[inner sep=2.5](51) at (240:2.5) {$(\sigma^2,1)$};
    \node[inner sep=2.5](60) at (180:2.5) {$(\sigma\tau,0)$};
     \draw[>=latex] (10)--(21);
     \draw[>=latex] (21)--(30);
     \draw[>=latex] (30)--(11);
     \draw[>=latex] (11)--(20);
     \draw[>=latex] (20)--(31);
     \draw[>=latex] (31)--(10);
     \draw[>=latex] (41)--(50);
     \draw[>=latex] (50)--(61);
     \draw[>=latex] (61)--(40);
     \draw[>=latex] (40)--(51);
     \draw[>=latex] (51)--(60);
     \draw[>=latex] (60)--(41);
     \draw[>=latex] (10)--(41);
     \draw[>=latex] (21)--(50);
     \draw[>=latex] (30)--(61);
     \draw[>=latex] (11)--(40);
     \draw[>=latex] (20)--(51);
     \draw[>=latex] (31)--(60);
      \draw (-2,2.5) node {$Y({\mathfrak S}_3, S_2)$};
     \end{scope}
     \end{scope}
    \end{tikzpicture}
    \caption{$X({\mathfrak S}_3,S_1), X({\mathfrak S}_3,S_2), Y({\mathfrak S}_3,S_1)$ and $Y({\mathfrak S}_3, S_2)$ with
    $\sigma=(1,2,3), \tau=(1,2)\in {\mathfrak S}_3$,
    We can observe the isomorphism $(\ref{eq:isomorphism1})$.
    }
    \label{fig:XY}
   \end{center}
  \end{figure}
 \end{center}
\end{exm}

\section{Random walks}
In the rest of the paper we assume the following:
 $a$ is a positive integer and $b$ is a non-negative integer 
 satisfying one of the following two conditions ${\bf \ref{cond:odd}}$
 and ${\bf \ref{cond:even}}$. 
 \begin{enumerate}[font=\bfseries, label=\arabic*]
  \item \label{cond:odd}
	$a+b$ is odd and $(a,b)\neq (2,1)$.
  \item  \label{cond:even}
	 $a$ is even  and $b$ is divisible by four.
 \end{enumerate}
The group $G$ is defined by
 \begin{equation}
\label{eq:Gcases}
 G = \begin{cases}
      {\mathfrak S}_{n} &
      {\bf \ref{cond:odd}} \mbox{~ holds},\\
      A_{n} &
      {\bf \ref{cond:even}} \mbox{~ holds},
     \end{cases}
 \end{equation}
where  $n=2a+b+1$.
The subsets $S_1$ and $S_2$ of $G$ are defined by $(\ref{eq:S1})$ and $(\ref{eq:S2})$ respectively.
Under these assumptions, by Theorem $\ref{thm:existence}$,
we have
\begin{equation}
\label{eq:covereq}
 Y(G, S_1) \cong Y(G, S_2).
\end{equation}
This section shows that the triple $(G, S_1, S_2)$
has the property ${\bf \ref{cond:RW}}$.

We first consider the simple random walk on $X(G,S)$
for an arbitrary subset $S$ of $G$,
which starts at the identity element ${\rm id}$ of $G$.
Namely we consider the sequence $(\mu^{(0)}, \mu^{(1)}, \mu^{(2)},\ldots)$
of the probability measures on $G$ defined as follows.
\[
 \mu^{(0)}(g) = \begin{cases}
		 1 & g = {\rm id},\\
		 0 & g \neq {\rm id},
		\end{cases}
\]
\[
 \mu^{(t+1)}(g) = \frac{1}{d}\sum_{k\in S}\mu^{(t)}(gk^{-1}),
\]
where $d=|S|$.
Therefore $\mu^{(t)}$ has another expression
\[
 \mu^{(t)}(g) = \frac{1}{d^t}\left|{\mathcal P}^{(t)}({\rm id}, g)\right|
\]
where ${\mathcal P}^{(t)}(g,h)$ is the set of paths from
$g$ to $h$ in $X(G,S)$.
If $X(G,S)$ is connected and aperiodic, then
$\mu^{(t)}$ converges to the uniform distribution $U_G$ on $G$,
\[
 \lim_{t\to\infty} \mu^{(t)}(g) = U_G(g) = \frac{1}{|G|},
\]
for all $g\in G$. (See \cite{levin2017markov} for the detail.)
The speed of the convergence is measured by the total 
variation distance ${\rm d}_{TV}(\mu^{(t)},U_G)$ defined
by
\[
 {\rm d}_{TV}(\mu^{(t)},U_G) = \frac{1}{2}\sum_{g\in G}\left|\mu^{(t)}(g)-U_G(g)\right|.
\]

For an  edge $e$  from $(g,c)$ to $(h,c+1)$
in $Y(G,S)$, we define $\pi_E(e)$ to be the edge
from $g$ to $h$ in $X(G,S)$.
Let $\pi_V:G\times C_2\to G$ be defined by
$\pi_V(g,c) = g$.
We state the following trivial fact as a lemma,
which is frequently used later.
\begin{lem}\label{lem:doublecover}
Let $\pi_V$ and $\pi_E$ be defined as above.
Then the pair $\pi=(\pi_V, \pi_E)$ 
is a double covering map 
from $Y(G,S)$ to $X(G,S)$. That is,
\begin{enumerate}
 \item $\pi$ is a graph morphism$:$
       $\pi_V(o(e))=o(\pi_E(e))$ and
       $\pi_V(t(e))=t(\pi_E(e))$,
       where $o(e)$ is the starting vertex of $e$ and
       $t(e)$ is the terminal vertex of $e$.
 \item $\pi_V$ and $\pi_E$ are two-to-one surjections.

 \item \label{prop:uniqlift}
       The restriction $\pi_E|_{E_x}$ of $\pi_E$ 
       is a bijection from $E_x$ to $E_{\pi_V(x)}$
       for every $x\in G\times C_2$,
       where $E_x$ is the set of edges starting from a vertex $x$.
\end{enumerate}
\end{lem}

By Lemma $\ref{lem:doublecover}$, $\pi_E$ can induce
the map ${\pi}_P$ from the set of paths in $Y(G,S)$ to those of $X(G,S)$.
That is, if $p=(e_1,e_2,\ldots,e_t)$ is a path
of length $t$ in $Y(G,S)$ where each $e_i$ is an edge,
then $\pi_P(p)=(\pi_E(e_1),\pi_E(e_2),\ldots,\pi_E(e_t))$
is a path in $X(G,S)$.
We define ${\mathcal P}_X^{(t)}(g,h)$ as the set of the paths of length $t$ 
starting from a vertex $g$ and terminating at a vertex $h$ in a graph $X$.
\begin{lem}
 \label{lem:pathbijection}
Let $X(G,S)$, $Y(G,S)$ and $\pi_P$ be defined as above.
Then the restriction of $\pi_P$ to
${\mathcal P}_{Y}^{(t)}((g,c),(h,[c+t]))$
is a bijection to
${\mathcal P}^{(t)}_{X}(g, h)$
for every $g\in G$,
where $[c+t]\in\{0,1\}$ satisfies
$[c+t]\equiv c+t( {\rm mod} 2)$.
\end{lem}
\begin{proof}
 Let $p=(e_1, e_2, \ldots, e_t)$
 be a path in $X(G,S)$, which corresponds to a sequence
 $(s_1, s_2, \ldots, s_t)$ of the elements of $S$. 
 That is, if $o(e_i)=g$ and $t(e_i)=h$,
 then $h=gs_i$.
 Then there exists a unique
 path $\bar{p}$ in $Y(G,S)$ 
 which starts from the vertex $(g,c)$
 and corresponds to the sequence
 $((s_1,1),(s_2,1),\ldots,(s_t,1))$ of the
 elements of $S\times\{1\}$.
 It is clear that this $\bar{p}$ is the unique path in $Y(G,S)$
 of length $t$
 such that $\bar{p}$ starts from $(g,c)$
 and $\pi_P(\bar{p})=p$.
 \end{proof}

\begin{lem}
 \label{lem:projection}
 Let $S$ be a generating set of a finite
 group $G$ and suppose that $S\times\{1\}$
 generates the direct product $G\times C_2$.
 Let $(\mu^{(0)}, \mu^{(1)},\ldots)$ be the
 simple random walk on $X(G,S)$
 starting from the identity element ${\rm id}$ of $G$,
 and
 let $(\nu^{(0)}, \nu^{(1)},\ldots)$ be
 the simple random walk on $Y(G, S)$
 starting from the identity element $({\rm id},0)\in G\times C_2$.
 Then,
 \begin{equation}\label{eq:projection2}
  \mu^{(t)}(g) = \nu^{(t)}(g,[t]),
 \end{equation}
 and therefore
 \begin{equation*}
  \nu^{(t)}(g,[t+1]) = 0.
 \end{equation*}
\end{lem}
\begin{proof}
It is clear that
\[
 \mu^{(t)}(g) = \frac{1}{d^t}\left|{\mathcal P}^{(t)}_{X}({\rm id},g)\right|,
 \mbox{~~ ~~}
 \nu^{(t)}(g,[t]) = \frac{1}{d^t}\left|{\mathcal P}_{Y}^{(t)}(({\rm id},0),(g,[t]))\right|.
\]
From Lemma $\ref{lem:pathbijection}$, $(\ref{eq:projection2})$ follows.
\end{proof}

By $(\ref{eq:covereq})$ $X(G, S_1)$ and $X(G, S_2)$
have a common double covering,
and there exists a graph isomorphism $\varphi=(\varphi_V, \varphi_E)$
from $Y(G,S_1)$ to $Y(G, S_2)$,
which satisfies $\varphi_V({\rm id}, 0)=({\rm id},0)$.
Then it is obvious that
\[
 \varphi_V\left(G\times\{0\}\right) = G\times\{0\},~~~
 \varphi_V\left(G\times\{1\}\right) = G\times\{1\},
\]
which induces two bijections $\varphi_0:G\to G$ and $\varphi_1:G\to G$ 
such that
\begin{equation}\label{eq:phi01}
 \varphi_V(g,0) = \left(\varphi_0(g),0\right),~~
 \varphi_V(g,1) = \left(\varphi_1(g),1\right). 
\end{equation}

\begin{thm}\label{thm:main}
 There exists a bijection between
 ${\mathcal P}^{(t)}_{X(G,S_1)}(g,h)$ and 
 ${\mathcal P}^{(t)}_{X(G,S_2)}(\varphi_0(g),\varphi_{[t]}(h))$
 and hence
 \begin{equation}\label{eq:main}
  \left|{\mathcal P}^{(t)}_{X(G,S_1)}(g,h)\right|
   =
   \left|{\mathcal P}^{(t)}_{X(G,S_2)}(\varphi_0(g),\varphi_{[t]}(h))\right|.
 \end{equation}
\end{thm}
\begin{proof}
 By Lemma $\ref{lem:pathbijection}$,
 there is a bijection between
 ${\mathcal P}^{(t)}_{X(G,S_1)}(g,h)$
 and
 ${\mathcal P}^{(t)}_{Y(G,S_1)}((g,0),(h,[t]))$.
 Since the condition $(\ref{eq:covereq})$ is satisfied, there exists a 
 graph isomorphism $\varphi: Y(G,S_1)\to Y(G,S_2)$.
 Let $\varphi_0$ and $\varphi_1$ be defined by $(\ref{eq:phi01})$.
 Then, we obtain the bijection
 \[
  {\mathcal P}^{(t)}_{Y(G,S_1)}((g,0),(h,[t]))
 \to
  {\mathcal P}^{(t)}_{Y(G,S_2)}(\varphi(g,0),\varphi(h,[t]))
 =
  {\mathcal P}^{(t)}_{Y(G,S_2)}((\varphi_0(g),0),(\varphi_{[t]}(h),[t])).
 \]
 By applying Lemma $\ref{lem:pathbijection}$,
 we obtain the bijection,
 \[
  {\mathcal P}^{(t)}_{Y(G,S_2)}((\varphi_0(g),0),(\varphi_{[t]}(h),[t]))
 \to
  {\mathcal P}^{(t)}_{X(G,S_2)}(\varphi_0(g),\varphi_{[t]}(h)).
 \]
 By composing these three bijections, we obtain a bijection
 \[
 {\mathcal P}^{(t)}_{X(G,S_1)}(g,h)\to{\mathcal P}^{(t)}_{X(G,S_2)}(\varphi_0(g),\varphi_{[t]}(h)).
 \]
\end{proof}

\begin{cor}
 Let $(\mu_i^{(0)},\mu_i^{(1)},\ldots )$ be the simple random
 walk on the Cayley graph $X(G,S_i)$ for $i=1,2$.
 Then, 
 \begin{equation}\label{eq:eqdist}
  \mu^{(t)}_1(g) = \mu_2^{(t)}(\varphi_{[t]}(g))
 \end{equation}
 for all $g\in G$ and $t\geq 0$, and
 \[
 {\rm d}_{TV}(\mu_1^{(t)}, U_G) = {\rm d}_{TV}(\mu_2^{(t)}, U_G),
 \]
 where $U_G$ is the uniform probability measure over $G$.
\end{cor}
\begin{proof}
 Since we suppose  the random walks start from the identity element ${\rm id}\in G$,
 \[
 \mu^{(t)}_1(g) = \frac{1}{d^t}\left|{\mathcal P}^{(t)}_{X(G,S_1)}({\rm id},g)\right|,
 ~~~
 \mu^{(t)}_2(\varphi_{[t]}(g)) = \frac{1}{d^t}\left|{\mathcal P}^{(t)}_{X(G,S_1)}({\rm id},\varphi_{[t]}(g))\right|.
 \]
 Since $\varphi_0({\rm id})={\rm id}$, Theorem $\ref{thm:main}$ implies
 \[
  \left|{\mathcal P}^{(t)}_{X(G,S_1)}({\rm id},g)\right| = 
 \left|{\mathcal P}^{(t)}_{X(G,S_1)}({\rm id},\varphi_{[t]}(g))\right|.
 \]
 Thus we obtain $(\ref{eq:eqdist})$.
 \begin{eqnarray*}
  {\rm d}_{TV}(\mu_1^{(t)},U_G) & = &
   \frac{1}{2}\sum_{g\in G}\left|\mu_1^{(t)}(g)-\frac{1}{|G|}\right|
   =  
   \frac{1}{2}\sum_{g\in G}\left|\mu_2^{(t)}\left(\varphi_{[t]}(g)\right)-\frac{1}{|G|}\right|\\
  & = & 
   \frac{1}{2}\sum_{g\in G}\left|\mu_2^{(t)}(g)-\frac{1}{|G|}\right|
   = 
  {\rm d}_{TV}(\mu_2^{(t)},U_G).
 \end{eqnarray*}
\end{proof}

\section{Spectral structure}
This section is about the properties ${\bf \ref{cond:nonisomorphic}}$ and ${\bf \ref{cond:spectral}}$.
Recall from Theorem $\ref{thm:existence}$ that ${\rm puz}_0(\theta_{a,b})$
stands for a connected component of ${\rm puz}(\theta_{a,b})$,
whose path contraction $\widetilde{{\rm puz}}_0(\theta_{a,b})$
is isomorphic to $Y(G, S_i)$ where $G$ is as in $(\ref{eq:Gcases})$,
and the generating sets $S_1$ and $S_2$ have been defined by $(\ref{eq:S1})$ and $(\ref{eq:S2})$ respectively. 
We apply the theory of the zeta functions of the finite graph coverings
to show that the Cayley graphs
$X(G, S_1)$ and $X(G, S_2)$ have
the properties 
${\bf \ref{cond:nonisomorphic}}$ and 
${\bf \ref{cond:spectral}}$.
The readers who are not familiar with the
theory of the zeta functions of the finite graphs are
referred to \cite{terras2010zeta}.
A remark is in order here.
\begin{rmk}
 \label{rmk:exception}
We will introduce the zeta function
and $L$-function of a finite graph $\Gamma$,
which is considered to be a {\em bi-directed} graph.
That is, every edge $e$ of $\Gamma$ is directed and
there exists a unique edge $\overline{e}$, the
{\em reverse} of $e$.
Among all the Cayley graphs of the forms $X(G, S)$ and $Y(G,S)$,
the only exception which does not fit into this formulation 
is our running example $\widetilde{{\rm puz}}_0(\theta_{1,0})$
since $X({\mathfrak S}_3, S_2)$ with $S_2=\{(1,2),(2,3),{\rm id}\}$ 
has loops and the loops does not have their reverses.
Nevertheless the statements of our main results Theorem $\ref{thm:spectra1}$ 
and $\ref{thm:spectra2}$ perfectly hold for this only exception
(shown in Figure $\ref{fig:theta10covering}$).
\end{rmk}

As we have explained in the previous section,
the automorphism group ${\rm Aut}(\theta_{a,b})$
contains a subgroup 
$K=\left\{{\rm id}, \rho, \psi, \rho\psi\right\}$
isomorphic to Klein four-group, which 
acts also on $Y=\widetilde{{\rm puz}}_0(\theta_{a,b})$
from the right 
as a subgroup of the automorphism group of $Y$,
namely we can regard as
\[
 K\subset{\rm Aut}(Y).
\]
The group $K$ has three non-trivial subgroup
$\langle\rho\rangle, \langle\psi\rangle$, and $\langle\rho\psi\rangle$.
Let $Y=\widetilde{{\rm puz}}_0(\theta_{a,b})=(V_Y, E_Y)$, where
$V_Y$ is the vertex set of  $Y$ and
$E_Y$ is the set of (directed) edges of $Y$.
For each subgroup $H$ of $K$,
we obtain the {\em quotient graph} $X_H=Y/H$ as follows.
The subgroup acts on $V_Y$ and $E_Y$ from the right.
The vertex set of $X_H$ is the set of $H$-orbits $V_Y/H$,
and the edge set of $X_H$ is the set of $H$-orbits $E_Y/H$.
Thus we obtain the natural covering map $\pi=(\pi_V, \pi_E)$
consisting of
\[
 \pi_V: V_Y \ni v \mapsto vH \in V_{X_H}=V_Y/H,~~~
 \pi_E: E_Y \ni e \mapsto eH \in E_{X_H}=E_Y/H.
\]
Then this {\em covering} is {\em normal} (or Galois),
that is,
$H$ acts transitively and freely on each fiber $\pi^{-1}_V(v)$
(and also on $\pi^{-1}_E(e)$).
When a graph $\Gamma$ is a normal cover of a graph $\Delta$,
the covering transformation group (or the Galois group) of $\Gamma/\Delta$
is denoted ${\rm Gal}(\Gamma/\Delta)$.
Thus we have
\[
 H = {\rm Gal}(Y/X_H).
\]

Let $Z=Y/K$.
Then ${\rm Gal}(Y/Z)=K$ and by applying Theorem 14.3 of \cite{terras2010zeta},
we have three intermediate coverings of $Y/Z$,
\[
 X_{\langle\rho\rangle} = Y/\langle\rho\rangle,~~~
 X_{\langle\psi\rangle} = Y/\langle\psi\rangle,~~~
 X_{\langle\rho\psi\rangle} = Y/\langle\rho\psi\rangle,
\]
each of which is a double (and therefore Galois)
covering of $Z$. Figure \ref{fig:galoiscovering} illustrates this covering
relation.
\begin{center}
 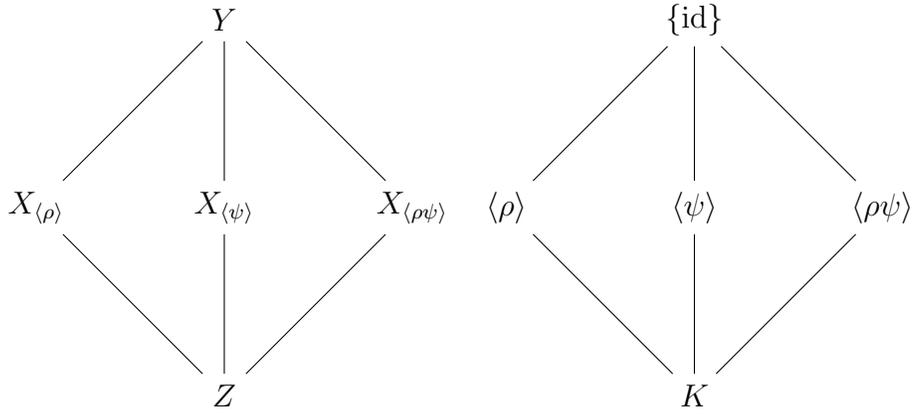
\begin{figure}[H]
  \begin{center}
   \begin{tikzpicture}[scale=2.5]
    \node[] at (0,0) (Z){$Z$} ;
    \node[] at (-1,1) (X1){$X_{\langle\rho\rangle}$} ;
    \node[] at (0,1) (X2){$X_{\langle\psi\rangle}$} ;
    \node[] at (1,1) (X3){$X_{\langle\rho\psi\rangle}$} ;
    \node[] at (0,2) (Y){$Y$} ;
    \draw (Z) to (X1);
    \draw (Z) to (X2);
    \draw (Z) to (X3);
    \draw (Y) to (X1);
    \draw (Y) to (X2);
    \draw (Y) to (X3);
    \begin{scope}[xshift=2.5cm]
    \node[] at (0,0) (K){$K$} ;
    \node[] at (-1,1) (G1){$\langle\rho\rangle$} ;
    \node[] at (0,1) (G2){$\langle\psi\rangle$} ;
    \node[] at (1,1) (G3){$\langle\rho\psi\rangle$} ;
    \node[] at (0,2) (id){$\{\rm id\}$} ;
    \draw (K) to (G1);
    \draw (K) to (G2);
    \draw (K) to (G3);
    \draw (id) to (G1);
    \draw (id) to (G2);
    \draw (id) to (G3);
    \end{scope}
   \end{tikzpicture}
  \end{center}
  \caption{Covering relation of $\widetilde{{\rm puz}}(\theta_{a,b})$, $X(G, S_1)$,
  $X(G, S_2)$, and $X_3, Z$
  }
  \label{fig:galoiscovering}
 \end{figure}
\end{center}

Note that we have
\[
 X_{\langle\rho\rangle} = X(G,S_1),~~~
 X_{\langle\psi\rangle} = X(G,S_2).
\]
The following is an example of $X_{\langle\rho\psi\rangle}$.
\begin{exm}
 As we have seen in Example $\ref{exm:theta1b}$,
 $\widetilde{{\rm puz}}(\theta_{1,0})$ is isomorphic to
 the Cayley graph $Y({\mathfrak S}_3, S_1)$ where
 $S_1=\{\sigma=(1,3,2), \tau=(1,3)\}$.
 Then the vertex permutation $\rho\psi$ is the transposition
 just exchanging the vertices $v_1$ and $v_3$.
 The quotient graph $X_{\langle\rho\psi\rangle}$ and $X_K$
 are shown in Figure $\ref{fig:X3}$.
 \begin{center}
  \begin{figure}[H]
   \begin{center}
    \begin{tikzpicture}[scale=1.25]
     \small
    \node[inner sep=2.5](10) at (120:1.3) {$(\sigma^{2},0)$};
    \node[inner sep=2.5](21) at (60:1.3) {$(\sigma,1)$};
    \node[inner sep=2.5](30) at (0:1.3) {$({\rm id},0)$};
    \node[inner sep=2.5](11) at (300:1.3) {$(\sigma^{2},1)$};
    \node[inner sep=2.5](20) at (240:1.3) {$(\sigma,0)$};
    \node[inner sep=2.5](31) at (180:1.3) {$({\rm id},1)$};
    \node[inner sep=2.5](41) at (120:2.5) {$(\tau\sigma,1)$};
    \node[inner sep=2.5](50) at (60:2.5) {$(\sigma\tau,0)$};
    \node[inner sep=2.5](61) at (0:2.5) {$(\tau,1)$};
    \node[inner sep=2.5](40) at (300:2.5) {$(\tau\sigma,0)$};
    \node[inner sep=2.5](51) at (240:2.5) {$(\sigma\tau,1)$};
    \node[inner sep=2.5](60) at (180:2.5) {$(\tau,0)$};
     \draw[>=latex] (10)--(21);
     \draw[>=latex] (21)--(30);
     \draw[>=latex] (30)--(11);
     \draw[>=latex] (11)--(20);
     \draw[>=latex] (20)--(31);
     \draw[>=latex] (31)--(10);
     \draw[>=latex] (41)--(50);
     \draw[>=latex] (50)--(61);
     \draw[>=latex] (61)--(40);
     \draw[>=latex] (40)--(51);
     \draw[>=latex] (51)--(60);
     \draw[>=latex] (60)--(41);
     \draw[>=latex] (10)--(41);
     \draw[>=latex] (21)--(50);
     \draw[>=latex] (30)--(61);
     \draw[>=latex] (11)--(40);
     \draw[>=latex] (20)--(51);
     \draw[>=latex] (31)--(60);
     \begin{scope}[xshift=4cm, yshift=2cm]
      \node[](w1) at (0,1.5) {$(\{{\rm id},\tau\},0)$};
      \node[](w2) at (2,1.5) {$(\{\sigma,\sigma\tau\},0)$};
      \node[](w3) at (4,1.5) {$(\{\sigma^2,\tau\sigma\},0)$};
      \node[](b1) at (0,-1.5) {$(\{{\rm id},\tau\},1)$};
      \node[](b2) at (2,-1.5) {$(\{\sigma,\sigma\tau\},1)$};
      \node[](b3) at (4,-1.5) {$(\{\sigma^2,\tau\sigma\},1)$};
      \draw (w1) -- (b1);
      \draw (w1) -- (b2);
      \draw (w1) -- (b3);
      \draw (w2) -- (b1);
      \draw (w2) -- (b2);
      \draw (w2) -- (b3);
      \draw (w3) -- (b1);
      \draw (w3) -- (b2);
      \draw (w3) -- (b3);
     \end{scope}

     \begin{scope}[xshift=6cm, yshift=-2cm]
      \node[](z1) at (0:1.2) {$\{{\rm id},\tau\}$};
      \node[](z2) at (120:1.2) {$\{\sigma, \sigma\tau\}$};
      \node[](z3) at (240:1.2) {$\{\sigma^2, \tau\sigma\}$};
      \draw (z1) -- (z2) -- (z3) -- (z1);
      \path (z1) edge[loop, looseness=10, ->, >=latex, in=60, out=-60] (z1);
      \path (z2) edge[loop, looseness=5, ->, >=latex, in=180, out=60] (z2);
      \path (z3) edge[loop, looseness=5, ->, >=latex, in=300, out=180] (z3);
     \end{scope}
    \end{tikzpicture}
   \end{center}
   \caption{$Y({\mathfrak S}_3, S_1)\cong Y=\widetilde{{\rm puz}}(\theta_{1,0})$(left),
   $X_{\langle\rho\psi\rangle}$(upper right)
   and $Z$(lower right)
   }
   \label{fig:X3}
  \end{figure}
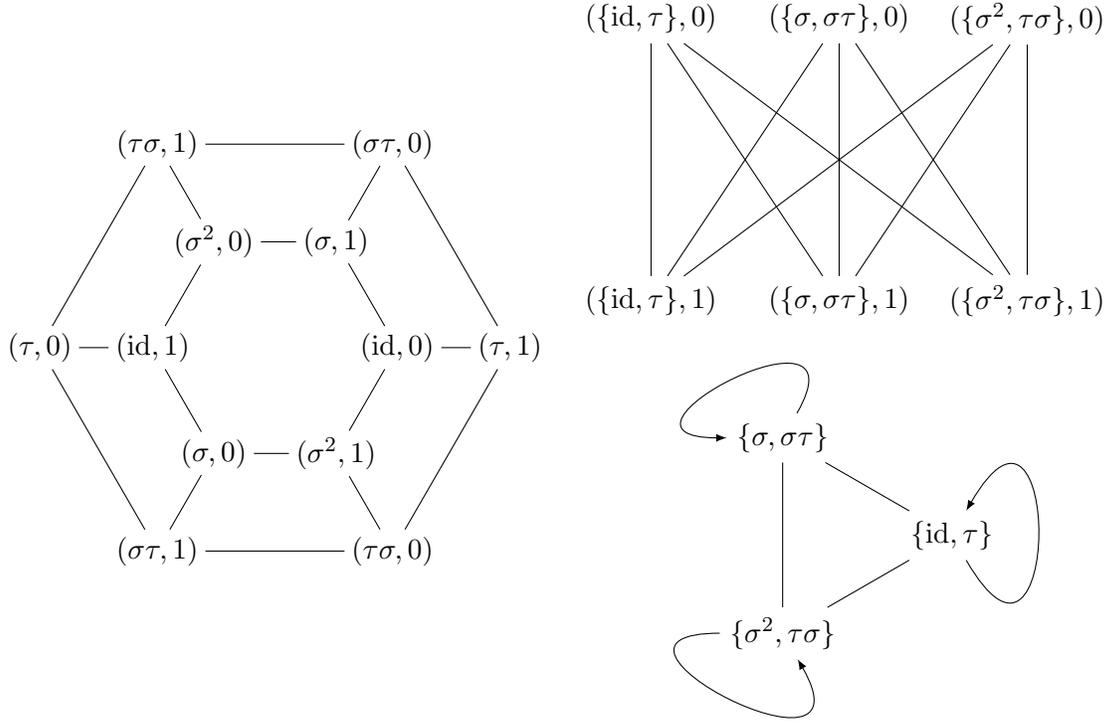
 \end{center}
\end{exm}

Let $C=(e_0,e_1,\ldots, e_{l-1})$ be a cycle in
a graph $\Gamma$.
Then $C$ is {\em non-backtracking} if
\[
 e_i\neq e_{i+1} 
\]
for every $i\in {\mathbb Z}/l{\mathbb Z}$.
The {\em length} $\nu(C)$ of a cycle $C$ is defined by
\[
 \nu(C) = l.
\]
A non-backtracking cycle $C$ is a {\em prime cycle} 
if there exists no pair of a cycle $D$ and an integer $k>1$ such that
\[
 C = D^k.
\]
Let $C=(e_0,e_1,\ldots, e_{l-1})$ be a cycle in a
graph $\Gamma$.
We  introduce an equivalence relation $\sim$ to the set of prime cycles
in $\Gamma$ as follows.
Let $C$ and $C'$ be two prime cycles in $\Gamma$.
Then we define that $C\sim C'$ if  there exists a $k\in {\mathbb Z}/l{\mathbb Z}$ such that
\[
 C' = (e_k, e_{k+1}, \ldots, e_{k+l-1}).
\]
That is, $C'$ is obtained from $C$ by changing the starting vertex.
This relation defines a equivalence relation on the set of all prime cycles in $\Gamma$,
and we denote the equivalence class $[C]$
to which $C$ belongs.
We call $[C]$ a {\em prime} in $\Gamma$.

The {\em zeta function} $\zeta_{\Gamma}$ of a graph 
$\Gamma$ is defined by
\[
 \zeta_{\Gamma}(u) = 
 \prod_{[C]}
 \left(1-u^{\nu(C)}\right)^{-1},
\]
where $[C]$ runs through the primes in $\Gamma$.
The following theorem shows how the zeta functions
are related to the spectra of $\Gamma$.
\begin{thm}
 \label{thm:zeta}
 {\rm (\cite{ihara1966discrete, bass1992ihara})}
 ~Let $A$ be the adjacency matrix of $\Gamma$, and
 let $Q$ be the diagonal matrix
 whose diagonal entry corresponding to the vertex $v$ is ${\rm deg}(v)-1$.
 Then we have
 \[
  \zeta_{\Gamma}(u)
 =
 \left(
 (1-u^2)^{r-1}
 \det\left(I - uA + u^2Q\right)\right)^{-1},
 \]
 where $r-1=\frac{1}{2}{\rm Tr}\left(Q-I\right)$.
\end{thm}
When we apply Theorem $\ref{thm:zeta}$ to $X_H = Y/H$,
we have $Q=2I$ and
\begin{eqnarray}
 \zeta_{X_H}(u) & = & (1-u^2)^{-|G|/|H|}u^{-2|G|/|H|}
 \det\left(\frac{2u^2+1}{u}I-A\right)^{-1}\\
 & = & (1-u^2)^{-|G|/|H|}u^{-2|G|/|H|} P_{X_H}\left(2u+\frac{1}{u}\right)^{-1},
 \label{eq:zetacharacteristic}
\end{eqnarray}
where $P_{\Gamma}(x)$ is the characteristic polynomial
of the adjacency matrix of a graph $\Gamma$.
We use this fact frequently to translate the relations
of the zeta functions into the relations of the characteristic polynomials.

If $\Gamma$ is a normal covering of a graph $\Delta$
where ${\rm Gal}(\Gamma/\Delta)$ is abelian,
we say  $\Gamma/\Delta$ is an {\em abelian} covering.
In what follows, we state theorems on Galois coverings
in \cite{terras2010zeta} in the forms restricted to the abelian cases.
Let $\Gamma/\Delta$ be an abelian covering with the
covering map $\pi=(\pi_V, \pi_E)$.
Let $C=(e_0, e_1, \ldots, e_{l-1})$ be a prime cycle of $\Delta$,
and $\widetilde{C}=(f_0,f_1,\ldots, f_{l-1})$ be a lift 
of $C$ in $\Gamma$, that is, $\widetilde{C}$ is a path in $\Gamma$
whose projection $\pi(\widetilde{C})$ onto $\Delta$ is $C$.
Let $o(\widetilde{C})$ be the starting vertex of $\widetilde{C}$ and
$t(\widetilde{C})$ the terminal vertex.
Then, since $\Gamma/\Delta$ is a normal covering, and
both $o(\widetilde{C})$ and $t(\widetilde{C})$ are projected onto the same vertex $v$, 
there exits a unique automorphism $g$ in ${\rm Gal}(\Gamma/\Delta)$ such
that $o(C)g = t(C)$. This unique automorphism $g$ is denoted
\[
\left(\frac{\Gamma/\Delta}{C}\right). 
\]
The following proposition corresponds to the parts $(1)$ and $(3)$
of Proposition $16.5$ of $\cite[p.137]{terras2010zeta}$
restricted to the abelian covering case.
\begin{prop}
The automorphism $\left(\frac{\Gamma/\Delta}{C}\right)$
is determined by the prime $[C]$ not depending on
the choice of $\widetilde{C}$.
\end{prop}
We call the automorphism $\left(\frac{\Gamma/\Delta}{C}\right)$
the {\em Frobenius automorphism} of $C$
associated with the abelian covering $\Gamma/\Delta$.
Let $\Gamma/\Delta$ be an abelian covering, and
let $\chi$ be a character of ${\rm Gal}(\Gamma/\Delta)$.
Then the {\em Artin $L$-function} $L(u, \chi, \Gamma/\Delta)$ 
of $\Gamma/\Delta$ associated
with the character $\chi$ is defined by
\begin{equation}
 \label{eq:Lfundef}
 L(u, \chi, \Gamma/\Delta) = 
 \prod_{[C]}\left(1-\chi\left(\frac{\Gamma/\Delta}{C}\right)u^{\nu(C)}\right)^{-1},
\end{equation}
where $[C]$ in the product runs through the primes of $\Delta$.
The following proposition corresponds to 
Proposition 18.10 and Corollary 18.11 of
\cite[p.154--155]{terras2010zeta}.
\begin{prop}
 \label{prop:factorization}
 Let $\Gamma/\Delta$ be an abelian covering,
 and $\widetilde{\Delta}$ be
 an intermediate covering of $\Gamma/\Delta$.
 Then, a character $\widetilde{\chi}$ of ${\rm Gal}(\widetilde{\Delta}/\Delta)$
 can be lifted to the character ${\chi}$
 of ${\rm Gal}(\Gamma/\Delta)$ and we have
 \begin{equation}
  \label{eq:Lfun1}
  L(u,\chi,\Gamma/\Delta) =  L(u, \widetilde{\chi}, \widetilde{\Delta}/\Delta).
 \end{equation}
 The zeta function $\zeta_\Gamma$ is factorized into the products of
 $L$-functions:
 \begin{equation}
  \label{eq:zetafactorization}
  \zeta_\Gamma(u)
 =
 \prod_{\chi\in \widehat{{\rm Gal}(\Gamma/\Delta)}}
 L(u, \chi, \Gamma/\Delta),
 \end{equation}
where $\widehat{{\rm Gal}(\Gamma/\Delta)}$ is the
set of the characters of ${\rm Gal}(\Gamma/\Delta)$.
\end{prop}
\begin{thm}
 \label{thm:spectra1}
Let $Y=\widetilde{{\rm puz}}(\theta_{a,b})$ with $(a,b)\neq (1,0)$.
Let $X_{\langle\rho\rangle}, X_{\langle\psi\rangle}, X_{\langle\rho\psi\rangle}$ and $Z$
be defined as above. Then
\[
 P_Y(u)P_Z(u)^2 = P_{X_{\langle\rho\rangle}}(u) P_{X_{\langle\psi\rangle}}(u) P_{X_{\langle\rho\psi\rangle}}(u)
\]
\end{thm}
\begin{proof}
We apply Proposition \ref{prop:factorization} to
our case where $\Gamma=Y=\widetilde{{\rm puz}}(\theta_{a,b})$
and $\Delta=Z=X_K$.
The Galois group  $K={\rm Gal}(Y/Z)=\{{\rm id}, \rho, \psi, \rho\psi\}$
has the four characters, which is summarized in the following table.
\begin{equation}
 \label{tab:character}
 \begin{array}{c|cccc}
  & {\rm id} & \rho & \psi & \rho\psi \\
  \hline
   1 & 1 & 1 & 1 & 1 \\
   \chi_\rho & 1 & 1 & -1 & -1 \\
   \chi_{\psi} & 1 & -1 & 1 & -1 \\
   \chi_{\rho\psi} & 1 & -1 & -1 & 1 
 \end{array}
\end{equation}
Then $\chi_\rho$ is the lift of the non-trivial character
$\widetilde{\chi}_\rho$
of $K/\langle\rho\rangle$ and it follows from $(\ref{eq:Lfun1})$ that
\[
 L(u, \chi_\rho, Y/Z) =
 L(u, \widetilde{\chi}_\rho, X_{\langle\rho\rangle}/Z).
\]
In the same manner we have
\[
 L(u, \chi_\psi, Y/Z) =
 L(u, \widetilde{\chi}_\psi, X_{\langle\psi\rangle}/Z),
\]
and
\[
 L(u, \chi_{\rho\psi}, Y/Z) =
 L(u, \widetilde{\chi}_{\rho\psi}, X_{\langle\rho\psi\rangle}/Z).
\]
Then by applying the factorization formula $(\ref{eq:zetafactorization})$ 
to abelian covers $X_{\langle\rho\rangle}/Z, X_{\langle\psi\rangle}/Z$ and $X_{\langle\rho\psi\rangle}/Z$,
we have
\begin{equation}
 \label{eq:zetaLrho}
 \zeta_{X_{\langle\rho\rangle}}(u) = \zeta_Z(u)L(u,\chi_\rho, Y/Z),
\end{equation}
\begin{equation}
 \label{eq:zetaLpsi}
 \zeta_{X_{\langle\psi\rangle}}(u) = \zeta_Z(u)L(u,\chi_\psi, Y/Z),  
\end{equation}
and
\[
 \zeta_{X_{\langle\rho\psi\rangle}}(u) = \zeta_Z(u)L(u,\chi_{\rho\psi}, Y/Z).
\]
Consequently, by applying $(\ref{eq:zetafactorization})$ to the abelian cover $Y/Z$,
we obtain
\begin{eqnarray*}
 \zeta_Y(u) & = & \zeta_Z(u) L(u,\chi_\rho, Y/Z) L(u,\chi_\psi, Y/Z) L(u,\chi_{\rho\psi}, Y/Z)\\
 & = & \zeta_Z(u)\frac{\zeta_{X_{\langle\rho\rangle}}(u)}{\zeta_Z(u)} \frac{\zeta_{X_{\langle\psi\rangle}}(u)}{\zeta_Z(u)}
  \frac{\zeta_{X_{\langle\rho\psi\rangle}}(u)}{\zeta_Z(u)},
\end{eqnarray*}
from which we have
\[
 \zeta_Y(u)\zeta_Z(u)^2 = \zeta_{X_{\langle\rho\rangle}}(u) \zeta_{X_{\langle\psi\rangle}}(u) \zeta_{X_{\langle\rho\psi\rangle}}(u).
\]
\end{proof}

\begin{lem}
 \label{lem:frobenius}
 Let $Y=\widetilde{\rm puz}_0(\theta_{a,b})$,
 let $Z = Y/K$ and
 let $C$ be a prime cycle of $Z$.
 Then we have
 \[
 \nu(C) \mbox{ is even } \Longleftrightarrow  \left(\frac{Y/Z}{C}\right) \in \left\{{\rm id}, \rho\psi\right\}.
 \]
\end{lem}
\begin{proof}
 As we have explained in Section ${\ref{sec:construction}}$,
 $Y\cong Y(G, S_i)$ is a bipartite graph.
 If two vertices (or positions) $f_1$ and $f_2$ are
 adjacent in $Y$, then the blanks of $f_1$ and $f_2$ do not coincide,
 that is,  $f_1^{-1}(0)\neq f_2^{-1}(0)$.
 Let $\widetilde{C}$ be the lift of $C$ to $Y$.
 If $\nu(C)$ is even and $\widetilde{C}$ starts
 from $f$ and terminate at $g$,
 then $f^{-1}(0) = g^{-1}(0)$ and $f\circ\left(\frac{Y/X}{C}\right) = g$,
 which means
 $\left(\frac{Y/X}{C}\right)\in \left\{{\rm id}, \rho\psi\right\}$
 since $\rho$ and $\psi$ move the blank.
 The converse is clear.
\end{proof}

\begin{thm}
 \label{thm:spectra2}
\begin{equation}
 \label{eq:charpolyrelation}
  \frac{P_{X_{\langle\rho\rangle}}(x)}{P_{Z}(x)}
 =
 (-1)^{|G|/2}\frac{P_{X_{\langle\psi\rangle}}(-x)}{P_{Z}(-x)}
\end{equation}
\begin{equation}
 \label{eq:charpolyrelation2}
  P_{X_{\langle\rho\psi\rangle}}(x)
 =
 (-1)^{|G|/2}P_{Z}(x)P_{Z}(-x)
\end{equation}
\end{thm}
\begin{proof}
 By $(\ref{eq:zetacharacteristic})$,$(\ref{eq:zetaLrho})$ and $(\ref{eq:zetaLpsi})$,
 to prove $(\ref{eq:charpolyrelation})$, it suffices to show
 \begin{equation}
  \label{eq:LeqRhoPsi}
  L(u, \chi_\rho, Y/Z) = L(-u, \chi_\psi, Y/Z).
 \end{equation}
 Let $C$ be a prime cycle in $Z$ and $\widetilde{C}$ its lift
 to $Y$.
 If $\nu(C)$ is an even integer, then by Lemma $\ref{lem:frobenius}$, we have
 \[
  \left(\frac{Y/Z}{C}\right) \in \{{\rm id}, \rho\psi\}.
 \]
 Further, by the table $(\ref{tab:character})$,
 we have $\chi_\rho\left(\frac{Y/Z}{C}\right)=\chi_\psi\left(\frac{Y/Z}{C}\right)$
 and the corresponding factors in $(\ref{eq:Lfundef})$
 of the $L$-functions coincide:
 \begin{equation}
  \label{eq:factoreq}
  1-\chi_\rho\left(\frac{Y/Z}{C}\right) u^{\nu(C)}
 =
  1-\chi_\psi\left(\frac{Y/Z}{C}\right) (-u)^{\nu(C)}.
 \end{equation}
 If $\nu(C)$ is an odd integer, then by Lemma $\ref{lem:frobenius}$, we have
 \[
  \left(\frac{Y/Z}{C}\right) \in \{\rho, \psi\}.
 \]
 Further, by the table $(\ref{tab:character})$, we obtain
 $\chi_\rho\left(\frac{Y/Z}{C}\right)=-\chi_\psi\left(\frac{Y/Z}{C}\right)$
 and $(\ref{eq:factoreq})$ holds. Thus we obtain
 $(\ref{eq:LeqRhoPsi})$.
 The relation $(\ref{eq:charpolyrelation2})$ can be obtained in the same manner.
\end{proof}

\begin{thm}
 Let $G,S_1$ and $S_2$ be as above.
 Then $X(G,S_1)$ is not isomorphic to $X(G, S_2)$.
\end{thm}
\begin{proof}
 As we have defined $X_{\langle\rho\rangle}\cong X(G, S_1)$ 
 (resp. $X_{\langle\psi\rangle}\cong X(G, S_2)$)
 as the quotient $Y/\langle\rho\rangle$
 (resp. $Y/\langle\psi\rangle$),
 if two vertices of $Y$ are on the same $\rho$-orbit 
 (resp. $\psi$-orbit), they are at odd distance.
 Hence paths connecting two vertices on $\rho$-orbit (resp. $\psi$-orbit)
 are projected onto cycles of odd length in $X_{\langle\rho\rangle}$
 (resp. $X_{\langle\psi\rangle}$).
 Thus $X_{\langle\rho\rangle}$ and $X_{\langle\psi\rangle}$
 are non-bipartite.
 For the same reason, $Z=Y/\langle\rho,\psi\rangle$ is also a non-bipartite graph
 containing prime cycles of odd length.
 Therefore we have
 \[
  L(u,\chi_{\rho}, Y/Z) \neq L(-u,\chi_{\rho}, Y/Z) = L(u,\chi_{\psi}, Y/Z).
 \]
 Hence
 \[
 \zeta_{X_{\langle\rho\rangle}}(u) = \zeta_{Z}(u)L(u,\chi_{\rho}, Y/Z)
 \neq
 \zeta_{Z}(u)L(u,\chi_{\psi}, Y/Z) 
 =
 \zeta_{X_{\langle\psi\rangle}}(u),
 \]
 which completes the proof.
\end{proof}

\begin{center}
 \begin{figure}[H]
\begin{center}
 \begin{tikzpicture}[scale=0.7]
  \small
  \begin{scope}[yshift=3cm]
  \foreach \theta in {0,60,...,300}{
  \draw[fill] (\theta:1.5) circle (0.05);
  \draw[fill] (\theta:3) circle (0.05);
  \draw (\theta:1.5) -- ({\theta+60}:1.5);
  \draw (\theta:1.5) -- ({\theta}:3);
  \draw (\theta:3) -- ({\theta+60}:3);
  }
  \draw (0,3.5) node {$Y=\widetilde{{\rm puz}}(\theta_{1,0})$};
  \draw (0,-4) node {$(x-3)(x-2)^2(x-1)x^4(x+1)(x+2)^2(x+3)$};
   \end{scope}

  \draw[line width = 2, gray] (0,-2) -- (0,-4);
  \draw[line width = 2, gray] (-3,-2) -- (-6,-4);
  \draw[line width = 2, gray] (3,-2) -- (6,-4);
  \draw[line width = 2, gray] (-6,-10) -- (-3,-13);
  \draw[line width = 2, gray] (6,-10) -- (3,-13);
  \draw[line width = 2, gray] (0,-10) -- (0,-13);

  \begin{scope}[xshift=-8cm, yshift=-7cm]
  \draw[fill] (0:1) circle (0.05) to ++(0:1);
  \draw[fill] (120:1) circle (0.05) to ++(120:1);
  \draw[fill] (240:1) circle (0.05) to ++(240:1);
  \draw[fill] (0:2) circle (0.05);
  \draw[fill] (120:2) circle (0.05);
  \draw[fill] (240:2) circle (0.05);
   \draw (0:1) -- (120:1) -- (240:1) -- cycle;
   \draw (0:2) -- (120:2) -- (240:2) -- cycle;
  \draw (0,2.5) node {$X_{\langle\rho\rangle}=X({\mathfrak S}_3, S_1)$};
  \draw (0,-2.5) node {$(x-3)(x-1)x^2(x+2)^2$};
  \end{scope}

  \begin{scope}[yshift=-7cm]
   \foreach \theta/\i in {0/0,60/1,120/2,180/3,240/4,300/5}{
   \draw[fill] (\theta:1.5) circle (0.05);
   \draw (\theta:1.5) -- ({\theta+60}:1.5);
   \node  at (\theta:1.5) (\i){};
   \path (\i) edge[loop, looseness=10, ->, >=latex, in={60+\theta}, out={-60+\theta}] (\i);
   }
  \draw (-1,2.5) node {$X_{\langle\psi\rangle}=X({\mathfrak S}_3, S_2)$};
  \draw (0,-2.5) node {$(x-3)(x-2)^2x^2(x+1)$};
  \end{scope}

  \begin{scope}[yshift=-7cm,xshift=7cm]
  \foreach \x in {0,2,4}{
  \draw[fill] (\x,1) circle (0.05);
  \draw[fill] (\x,-1) circle (0.05);
  \draw (\x,1) -- ++(0,-2);
  }
   \draw (0,1) -- ++(2,-2);
   \draw (0,1) -- ++(4,-2);
   \draw (2,1) -- ++(-2,-2);
   \draw (2,1) -- ++(2,-2);
   \draw (4,1) -- ++(-2,-2);
   \draw (4,1) -- ++(-4,-2);
  \draw (0,2.5) node {$X_{\langle\rho\psi\rangle}$};
  \draw (2,-2.5) node {$(x-3)x^4(x+3)$};
  \end{scope}

  \begin{scope}[yshift=-16cm]
   \draw[fill] (0:1.5) circle (0.05);
   \draw[fill] (120:1.5) circle (0.05);
   \draw[fill] (240:1.5) circle (0.05);
   \node  at (0:1.5) (0){};
   \node at (120:1.5) (1){};
   \node at (240:1.5) (2){};
   \draw (0:1.5) -- (120:1.5) -- (240:1.5)--cycle;
   \path (0) edge[loop, looseness=10, ->, >=latex, in=60, out=-60] (0);
   \path (1) edge[loop, looseness=10, ->, >=latex, in=180, out=60] (1);
   \path (2) edge[loop, looseness=10, ->, >=latex, in=300, out=180] (2);
  \draw (-1,2.5) node {$Z=X_K$};
  \draw (0,-2.5) node {$(x-3)x^2$};
  \end{scope}
 \end{tikzpicture}
 \caption{
 The statements of Theorem $\ref{thm:spectra1}$ and $\ref{thm:spectra2}$
 hold for this exceptional case $\theta_{1,0}$. 
 }
 \label{fig:theta10covering}
\end{center}
 \end{figure}
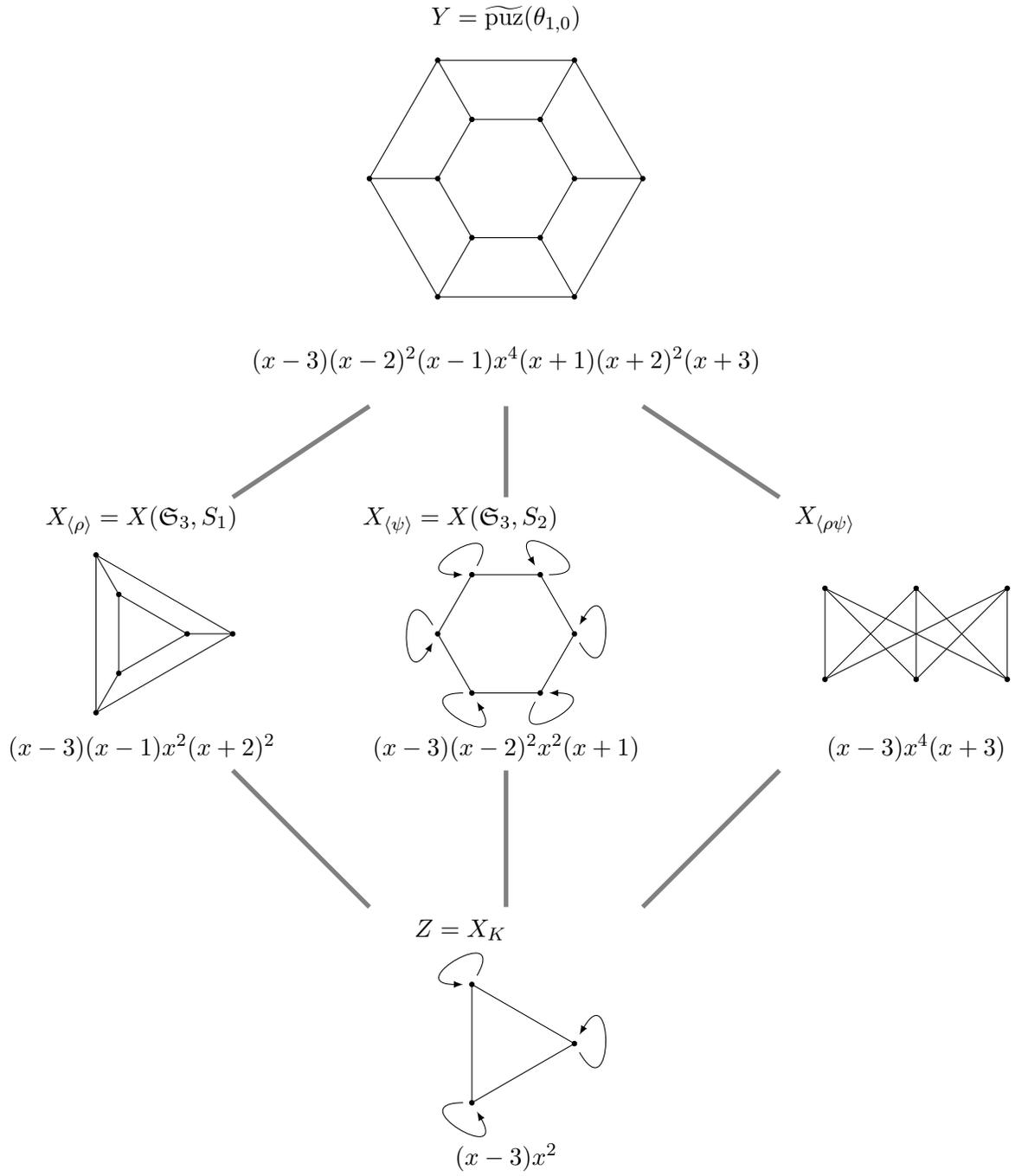
\end{center}

\section*{Acknowledgments}
This research was  partially supported by JSPS KAKENHI Grant Number JP20K03558
and JP20K03659.


\end{document}